\documentclass[12pt]{article}
\usepackage{preprint-layout} 
\usepackage{preprint-notation} 
\usepackage[]{preprint-tools} 
\newcommand{\hatM}{\widehat{M}}
\newcommand{\hatA}{\widehat{A}}
\newcommand{\hatB}{\widehat{B}}

\newcommand{\hatE}{\widehat{E}}
\newcommand{\IIh}{\mcI_{h,I}}
\newcommand{\bur}[1]{{\color{black}#1}}
\newcommand{\mgl}[1]{{\color{black}#1}}

\DeclareMathOperator{\Cyl}{Cyl}

\title{Explicit Time Stepping for the Wave Equation using CutFEM with Discrete Extension}
\date{\today}

\author{
Erik Burman,
Peter Hansbo,
Mats G. Larson 
}
\date{}

\begin{document}

\maketitle

\begin{abstract}
In this note we develop a fully explicit cut finite element method for
the wave equation. The method is based on using a standard leap frog scheme 
combined with an extension operator that defines the nodal values outside of the 
domain in terms of the nodal values inside the domain. We show that the mass 
matrix associated with the extended finite element space can be lumped leading 
to a fully explicit scheme. We derive stability estimates for the method and provide 
optimal order a priori error estimates. Finally, we present some illustrating numerical 
examples.
\end{abstract}

\section{Introduction}

\paragraph{New Contributions.}
Let $\Omega \subset \mathbb{R}^d$, with $d\geq 2$  be an open connected domain 
with smooth boundary $\Gamma$. We consider the wave equation: find $u:[0,T)
\rightarrow H^2(\Omega)$ such that 
\begin{equation}\label{eq:wave}
\frac{\partial^2 u }{\partial t^2} - \Delta u = f \quad \text{in $(0,T) \times \Omega$},
\qquad u = 0\quad  \text{on $(0,T) \times \Gamma$} 
\end{equation}
with initial data $u= u_0$ and $\partial u/\partial t = u_1$ at $t=0$, and right hand 
side $f:[0,T) \rightarrow L^2(\Omega)$.
The objective of the present note is to design an explicit cut finite element 
method for the approximation of solutions to \eqref{eq:wave}. The method 
uses a leapfrog scheme for the time discretisation combined with an extension 
operator which provides values in nodes outside of the domain in terms of the 
interior nodal values. The extension is based on a composition of an extension 
operator from interior elements into the space of discontinuous piecewise polynomials 
and an average operator that projects into the continuous finite element space. 
The framework is quite general, allows for several natural implementations, is 
convenient for analysis, and may be viewed as a generalization of previous 
constructions, see \cite{BadVer18}. We prove stability and 
interpolation results for the extended finite element space. To 
construct a purely explicit scheme we show that the mass matrix associated with 
the extended finite element space can indeed be lumped while preserving optimal 
order for piecewise linear elements. Key to this result is the fact that the elements 
in the mass matrix associated with the extended finite element space are all non 
negative, which is not the case for popular stabilization procedures such as
stabilization of the jump in derivatives across faces. 

Combining cut finite elements, the extension operator, and mass lumping we obtain a 
very simple fast explicit method which can handle complex geometric situations thanks to the flexibility provided by the cut finite element method.

We note that the discrete extension operator provides an alternative to weak 
stabilization of the cut elements through the bilinear form which controls jumps 
in derivatives across faces. The extension operator is therefore of interest in its own 
right and may find other applications, for instance, for the computation of physical 
fluxes in the shifted boundary method. Furthermore, our construction and theory 
of the extension operator extends to higher order polynomials. Since our focus is 
on explicit lumped methods we restrict the presentation to piecewise linears.

\paragraph{Previous Work.} Cut finite elements allow the boundary of the domain to 
cut through an underlying fixed mesh in an arbitrary manner. This procedure manufactures 
so called cut elements in the vicinity of the boundary that may lead to stability problems 
and bad conditioning of the resulting algebraic equations. The remedy is to add some form 
of stabilization for instance a weak least squares control on the jump in the normal gradient 
across element faces, so called ghost penalty, see \cite{Bur10,BurHan12,HanLar17,MasLar14} for various applications of this concept. Another approach 
to handle cut elements is to eliminate them using agglomeration 
where small elements are connected to larger elements in order to form an element with a 
sufficiently large intersection with the domain, see \cite{JohLar13} for a discontinuous method, 
and \cite{BadVer18} for an extension operator where degrees of freedom associated with external nodes are eliminated using a local average of internal node values. For a general 
introduction to cut finite element methods we refer to the overview article \cite{BurCla15}.

Error analysis of finite element methods for the wave equation was originally developed in
early papers including, \cite{Dup73,Bak76,BakDou76}, space time methods 
were proposed and analysed in \cite{HulHug84} and \cite{Joh93}.  Recent works on wave 
equations focus on explicit schemes \cite{DiaGro09,DroMoh20} and discontinuous 
Galerkin methods  \cite{GroSch06,GroSch09}. Cut finite element methods for the 
wave equation were developed in \cite{StiKre19} and \cite{StiLud20}, in particular the authors 
consider higher order elements with face stabilization combined with an explicit Runge-Kutta 
time stepping scheme which involves inversion of the mass matrix.

\paragraph{Outline.} In Section 2 we first introduce the discrete extension operator 
and derive stability estimates and interpolation error bounds for the extended finite element space. Then we formulate the finite element method. In Section 3 we prove a stability 
estimate for the method and then we prove optimal order a priori error estimates taking 
also lumping of the mass matrix into account. Finally, in Section 4 we present illustrating numerical examples.

\section{The Finite Element Method}
\subsection{Standard Notation}
We shall use the following standard notation. $H^s(\omega)$ denotes the Sobolov spaces of order $s$ over the set $\omega$ with norm $\| \cdot \|_{H^s(\omega)}$. For $s=0$ we write $L^2(\omega) = H^0(\omega)$ and $\| \cdot \|_{L^2(\omega)} = \| \cdot \|_\omega$. 
In the case $\omega = \Omega$ we simplify further and write $\| \cdot \|_{L^2(\Omega)} = \| \cdot \|$. The $L^2(\omega)$ inner product is denoted by  $(v,w)_\omega = \int_\omega v w$ and for $\omega = \Omega$ we write $(v,w)_\Omega = (v,w)$.

\subsection{Mesh and Finite Element Spaces}
We introduce the following notation:
\begin{itemize}

\item We let $\Omega_0$ be a polygonal domain with $\Omega \subset
\Omega_0$ and assume that $\mathcal{T}_{0,h}$ is a quasi uniform triangulation
of $\Omega_0$ with mesh parameter $h \in (0,h_0]$ for some $h_0>0$. We let $\mcT_h$ denote the active mesh $\mcT_h = \{ T \in \mcT_{h,0} : T \cap \Omega \neq \emptyset\}$.
We let $\mathcal{F}_h$ denote the set of interior faces in $\mathcal{T}_h$. 

\item We let $\mathcal{X}_h$ be the set of vertices in $\mathcal{T}_h$ 
and denote \bur{its} cardinality by $N_h$.

\item We define the space of piecewise linear discontinuous functions 
$W_h$ on $\mathcal{T}_h$  and the subspace of continuous piecewise linear 
functions $V_h:=W_h \cap C^0(\Omega_h)$, where $\Omega_h =\cup_{T \in \mcTh} T$.

\item We shall often use scalar products and norms defined on a set of mesh entities. For instance, let $\widetilde{\mcT}_h \subset \mcT_h$ be a subset of elements then 
\begin{equation}
(v,w)_{\widetilde{\mcT}_h} = \sum_{T\in \widetilde{\mcT}_h} (v,w)_T, 
\qquad \| v \|^2_{\widetilde{\mcT}_h} = \sum_{T\in \widetilde{\mcT}_h} \| v \|^2_T
\end{equation}

\end{itemize}

\subsection{Discrete Extension}

It is well known \cite{Ste70}, Theorem 5, page 181, \bur{that for domains
with sufficiently smooth boundary}, there exists a universal stable extension
operator $E:H^s(\Omega) \mapsto H^s(\mathbb{R}^d)$, \bur{$s \in \mathbb{N}_+$},
\begin{equation}\label{eq:ext-cont-stab}
\|E u\|_{H^s(\mathbb{R}^d)} \lesssim  \| u\|_{H^s(\Omega)}, 
\end{equation}
We will now  construct a stable discrete extension operator. The construction is based on 
polynomial extension into the discontinuous finite element space $W_h$ and then application of an average operator to obtain a continuous piecewise linear function in $V_h$. We first 
recall such an average operator $A_h$. 

\paragraph{Average Operator.}
Let the nodal averaging operator $A_h:W_h \mapsto V_h$ be defined by
\begin{equation}\label{eq:averege-operator}
A_h: W_h \ni w \mapsto \sum_{x \in \mcX_h}  \langle w \rangle_{x}  \varphi_x  \in V_h
\end{equation}
where the average of the discontinuous function $w\in W_h$ at a node $x\in \mcX_h$ 
is defined by 
\begin{equation}\label{eq:average-nodal}
\langle w \rangle_x 
= 
 \sum_{T \in \mcT_h(x)}  \kappa_{T,x} w|_T (x)
\end{equation}
where the weights $\kappa_{T,x}$ satisfy
\begin{equation}
\kappa_{T,x}\geq 0, \qquad \sum_{T \in \mcT_h(x)}  \kappa_{T,x} = 1
\end{equation}
and $\mcT_{h}(x) = \{ T \in \mcT_h : x \in T\}$ with cardinality  $|\mcT_{h}(x)|$. 
We have the following estimate see \cite{BurErn07},
\begin{equation}\label{eq:Ah-error-est}
\|w - A_h w \|_{\mcTh}  \lesssim h^{1/2} \| [ w ] \|_{\mcFh} 
\end{equation}
For completeness we include a brief derivation.
\begin{proof}[Proof of (\ref{eq:Ah-error-est})] Letting $w_T = w|_T$ and using an 
inverse estimate to pass from the elements to the nodes we obtain
\begin{align}
&\| w - A_h w \|^2_{\mcTh} = \sum_{T\in \mcTh}  \| w_T - A_h w \|^2_T
\lesssim
\sum_{T\in \mcTh} h^d \| w_T - A_h w \|^2_{\mcX_h(T)} 
\\
&\qquad
\lesssim 
\sum_{T\in \mcTh} \sum_{x \in \mcX_h(T)} h^d |w_T(x) - \langle w \rangle_x |^2
\lesssim
\sum_{T\in \mcTh} \sum_{x \in \mcX_h(T)} \sum_{F \in \mcF_h(x)} 
h \|[w]\|^2_F
\lesssim h   \|[w]\|^2_{\mcFh}
\end{align}
where $\mcX_h(T)$ is the nodes associated with $T$, $\mcF_h(x)$ the 
faces belonging to node $x\in \mcX_h$, and we finally used the inverse estimate
\begin{equation}\label{eq:Ah-inverse}
\sum_{T \in \mcT_h(x)} |w_T(x) - \langle w \rangle_x |^2 
\lesssim 
\sum_{F \in \mcF_h(x)} 
h \|[w]\|^2_F
\end{equation}
To establish (\ref{eq:Ah-inverse}) we note, using the fact that the weights in the average sum to one, that
\begin{equation}\label{eq:Ah-a}
\sum_{T \in \mcT_h(x)} |w_T(x) - \langle w \rangle_x |^2 
=
\sum_{T,S \in \mcT_h(x)} \kappa^2_{S,x}| w_S(x)  - w_T(x) |^2 
\lesssim
\sum_{F \in \mcF_h(x)} | [ w(x) ] |^2
\end{equation}
We complete the argument using the inverse estimate $|v(x)|^2 \lesssim h^{d-1} \| v \|^2_F$ 
with $v = [w]$.
\end{proof}

\paragraph{Extension Operator.} To define the extension operator  we split $\mcTh$ as follows
\begin{equation}
\mcTh = \mcT_{h,B} \cup \mcT_{h,I}
\end{equation}
where $\mcT_{h,I}$ is the set of elements in the interior of $\Omega$ (or with sufficiently large intersection with $\Omega$ see Remark \ref{rem:eh}) and 
$\mcT_{h,B}$ are the elements that intersect the boundary, 
\begin{equation}
\mcT_{h,I} = \{ T \in \mcT_h : T \subset \Omega\},\qquad \mcT_{h,B } = \mcTh \setminus \mcT_{h,I}
\end{equation}
Let $W_{h,I} = W_h |_{\mcT_{h,I}}$ and $V_{h,I} = V_h |_{\mcT_{h,I}}$. We construct
an extension operator $F_h : W_{h,I} \rightarrow F_h W_{h,I} \subset W_h$ by using 
canonical polynomial extensions from a nearest neighbouring element $T \in \mcT_{h,I}$. Restricting $F_h$ to $V_{h,I}$ and composing with the average operator 
$A_h$ we obtain a discrete extension operator $E_h : V_{h,I} \rightarrow E_h V_{h,I} 
\subset V_h$. \bur{The space $E_h V_{h,I}$ will be our approximation
  space and we will use the notation 
\begin{equation}
V_h^E = E_h V_{h,I}
\end{equation}  
Observe that $V_h^E$ is a proper subspace of $V_h$, however as we shall see under mild 
assumptions on the mesh geometry it has similar approximation properties.}

\bur{To make things precise}, let $S_h: \mcT_{h,B}\rightarrow \mcT_{h,I}$ be a mapping that associates 
an element $T \in \mcT_{h,I}$ with each element $T\in \mcT_{h,B}$ and assume that 
there is a constant such that for all  $h \in (0,h_0]$ and $T \in \mcT_{h,B}$, 
\begin{equation}\label{eq:assumption}
\text{diam}(T\cup S_h(T)) \lesssim h
\end{equation}
For $h_0$ small enough there is such a mapping $S_h$, see Lemma \ref{lem:Nh-map} below. 
We extend $S_h$ from $\mcT_{h,B}$ to $\mcT_h$ by letting $S_h(T) = T$ for $T\in \mcT_{h,I}$.

For $v \in \mathbb{P}_1(T)$ we let $v^e \in \mathbb{P}_1(\IR^d)$ denote the canonical extension such that $v^e|_T = v$. We can then define the discrete extension operator 
$F_h : W_{h,I} \rightarrow W_h$ as 
follows
\begin{align}\label{eq:Fh}
(F_h v )|_T = (v|_{S_h(T)})^e|_T 
\end{align}
and then define the discrete extension operator $E_h : V_{h,I} \rightarrow V_h$,
\begin{equation}
E_h = A_h \circ F_h
\end{equation}

\begin{rem}\label{rem:eh} In practice, we can define the set of elements that have a large intersection 
with the domain as follows,
\begin{equation}
\mcT_{h,\text{large}} = \{ T \in \mcTh : |T \cap \Omega| \geq c h^d \}
\end{equation}
for some positive constant $c$. Then for small enough $c$ we have 
$\mcT_{h,I} \subset \mcT_{h,large}$ and we can define the mapping 
$S_h : \mcT_h \setminus \mcT_{h,large} \rightarrow \mcT_{h,large}$. This approach has the 
advantage that fewer elements are mapped resulting in a simpler map $F_h$.
\end{rem}

\begin{rem} The construction of the extension operator and the forthcoming theory directly extends to higher order polynomials.
\end{rem}

We will now prove that the extension is stable and that the associated interpolation operator has optimal approximation properties.

\begin{lem} \label{lem:Nh-map}
For $h_0$ small enough there is a mapping $S_h:\mcTh \rightarrow \mcT_{h,I}$ that satisfies (\ref{eq:assumption}).
\end{lem}
\begin{proof} Note first that there is $\delta_0>0$ such that the closest point mapping 
$p : U_\delta(\Gamma) \rightarrow \Gamma$ is well defined for $\delta \in (0,\delta_0]$. 
For $T \in \mcT_{h,B}$ take $x \in T \cap \Gamma$ and let $T_x(\Gamma)$ be the tangent 
plane to $\Gamma$ at $x$ with exterior unit normal $n_x$. Let $\rho_{T_x(\Gamma)}$ be a signed distance function associated with $T_x(\Gamma)$ such that $\nabla \rho_{T_x(\Gamma)} = - n _x$ and define 
the one sided tubular neighborhood  $U^+_\delta(T_x(\Gamma)) = \{ y \in \IR^d : 0< \rho_{T_x(\Gamma)} < \delta\}$. Then we note that there is a fixed $\delta_1$ such that 
for all $\delta \in (0,\delta_1]$, 
\begin{equation}
O_\delta(x) = (U^+_{\delta}(T_x(\Gamma)) \setminus U^+_{c \delta^2} (T_x(\Gamma)) ) \cap \Cyl_\delta(x,n_x) \subset 
U_{\delta_0} \cap \Omega 
\end{equation}
where $\text{Cyl}_\delta(x,n_x)$ is the cylinder with radius $\delta$ and center axis aligned with 
the normal $n_x$ at $x \in \Gamma$.
Taking $\delta$ such that $c h \leq \delta \leq C h$, with $c$ and $C$ sufficiently large constants, we conclude that there is an element $T \in \mcT_{h,I}$ such that $T \subset O_\delta(x)$ for $h \in (0,h_0]$, with $h_0$ small enough to guarantee that $O_\delta(x) \subset U_{\delta_0}(\Gamma)$ for $\delta = C h_0$.
\end{proof}

\begin{lem} There are constants such that for all $v \in V_{h,I}$,
\begin{equation}\label{eq:Fh-stab-L2}
\| F_h v \|_{\mcTh} \lesssim \| v \|_{\mcT_{h,I}} 
\end{equation}
\begin{align}\label{eq:Fh-stab-energy}
\bur{\| \nabla F_h v \|_{\mcTh} 
+ h^{-1} \|[ F_h v ] \|_{\mcFh}
\lesssim \| \nabla v \|_{\mcT_{h,I}} }
\end{align}
\end{lem}
\begin{proof} To prove (\ref{eq:Fh-stab-L2}) we note that for each $T\in \mcT_{h,B}$ we 
have the inverse inequality  
\begin{equation}\label{eq:Fh-stab-a}
\| v^e \|_T \leq \| v^e \|_{B_\delta} \lesssim \| v \|_{S_h(T)}
\end{equation}
where $B_\delta$ is a ball with diameter $\delta \sim h$ such that $T
\cup S_h(T) \subset B_\delta$.  Summing over $T \in \mcT_{h,B}$ \mgl{
  and noting that thanks to \eqref{eq:assumption}, the number of elements in $\mcT_{h,B}$ that $S_h$ 
  maps to $T$ is uniformly bounded over all $T\in \text{Im}(S_h)$,   
 \begin{equation}
 \sum_{T \in \mcT_{h,B}} \| v \|^2_T \lesssim  \sum_{T \in \mcT_{h,B}}  \| v \|^2_{S_h(T)} \lesssim 
  \sum_{T \in \text{Im}(S_h) } \| v \|^2_T 
  \lesssim \| v \|^2_{\mcT_{h,I}} 
 \end{equation}
where for the last inequality we used the inclusion $\text{Im}(S_h) \subset \mcT_{h,I}$.
}
For (\ref{eq:Fh-stab-energy}), we obtain using the same argument 
\begin{equation}
\| \nabla F_h v \|_{\mcTh} 
\lesssim \| \nabla v \|_{\mcT_{h,I}}
\end{equation}
To estimate the remaining term 
\begin{equation}
h^{-1} \|[ F_h v ] \|^2_{\mcFh} 
= \sum_{F \in \mcFh} h^{-1} \|[ F_h v ] \|_F^2
\end{equation}
we have for each $F\in \mcF_h$, $[v] = [v- w_F]$ for an arbitrary constant $w_F$. Using 
the triangle inequality followed by an inverse inequality to pass from the face $F$ to the elements $\mcT_h(F)$ sharing $F$,
\begin{equation}
h^{-1} \|[ F_h v ] \|^2_F 
\lesssim h^{-2} \|F_h v - w_F \|^2_{\mcT_h(F)} 
\lesssim  h^{-2} \|v - w_F \|^2_{S_h(\mcT_h(F))} 
\end{equation}
Next there is an open ball $B_\delta$ with diameter $\delta \sim h$ such that 
\begin{align}
S_h (\mcT_h(F)) \subset B_\delta 
\end{align}
and then we have 
\begin{equation}
h^{-2} \bur{\inf_{w_F \in \mathbb{R}}} \|v - w_F \|^2_{S_h(\mcT_h(F))}  
\leq 
h^{-2} \bur{\inf_{w_F \in \mathbb{R}}} \|v - w_F \|^2_{\mcT_{h,I} (B_\delta)} 
\lesssim 
\delta^2 h^{-2} \| \nabla v \|^2_{\mcT_{h,I} (B_\delta)} 
\lesssim 
\| \nabla v \|^2_{\mcT_{h,I} (B_\delta)} 
\end{equation}
which concludes the proof. 
\end{proof}

\bur{A key property of CutFEM stabilized using ghost penalty is that
  the weakly consistent penalty term allows for control of the finite
  element solution on the whole mesh domain, by the combination of the
  stability from coercivity on the physical domain and the penalty
  terms. We will now show that such a stability property holds by
  construction for the extended space, thereby eliminating the need for
  additional stabilization.}

\begin{lem}\label{lem:stab_ext}(Stability of the extension)
There are constants such that for all $v \in V_{h,I}$,
\begin{equation}\label{eq:Eh-stab}
\|\nabla^m E_h v_h\|_{\mathcal{T}_h} \lesssim  \|\nabla^m v_h\|_{\mcT_{h,I}}, \qquad m=0,1
\end{equation}
\end{lem}

\begin{proof}
For $m=0$, we add and subtract $F_h v$ and use (\ref{eq:Fh-stab-L2}) and 
(\ref{eq:Ah-error-est}) to conclude that
\begin{align}
\| E_h v \|_{\mcTh} &= \| A_h F_h v \|_{\mcTh}
\\
&\leq \| F_h v \|_{\mcTh} + \| (I-A_h) F_h v \|_{\mcTh}
\\
&\lesssim \| F_h v \|_{\mcTh} + h^{1/2} \| [ F_h v ] \|_{\mcFh}
\\
&\lesssim \| F_h v \|_{\mcTh} +  \| F_h v  \|_{\mcTh}
\\
&\lesssim \| v \|_{\mcT_{h,I}}
\end{align}
For $m=1$, we proceed in the same way but we instead employ the stronger 
stability (\ref{eq:Fh-stab-energy}) of the operator $F_h$,
\begin{align}
\| \nabla E_h v \|_{\mcTh} &= \| \nabla A_h F_h v \|_{\mcTh}
\\
&\leq \| \nabla F_h v \|_{\mcTh} + \| \nabla (I - A_h)  F_h v) \|_{\mcTh}
\\
&\leq \| \nabla F_h v \|_{\mcTh} + h^{-1} \| (I -  A_h) F_h v \|_{\mcTh}
\\
&\lesssim \| \nabla F_h v \|_{\mcTh} + h^{-1/2} \| [ F_h v ] \|_{\mcFh}
\\
&\lesssim \| \nabla v \|_{\mcT_{h,I}}
\end{align}
and thus the proof is complete.
\end{proof}

\subsection{Interpolation}

We begin by defining some interpolation operators that will be needed in the analysis.
\begin{itemize}
\item Let $\pi_h:H^1(\Omega_h) \rightarrow V_h$ be an interpolation operator of average 
type, see \cite{Cle75} or \cite{ScoZha90}, that 
satisfies the standard element wise estimate 
\begin{equation}\label{eq:interpol-local}
\| v - \pi_h v \|_{H^m(T)} \lesssim h^{2-m} \| v \|_{H^2(\mcT_h(T))}, \qquad m=0,1
\end{equation} 
with $\mcT_h(T)\subset \mcTh$ the neighboring elements of $T$. Composing $\pi_h$ 
with the continuous extension operator $E$ we obtain an interpolation operator 
$\pi_h \circ E : H^1(\Omega) \rightarrow V_h$ and using the 
stability (\ref{eq:ext-cont-stab}) of the continuous extension operator we have 
\begin{equation}\label{eq:interpol}
\| E v - \pi_h E v \|_{\mcTh} \lesssim h^{2-m} \| v \|_{H^2(\Omega_h)}
\lesssim 
h^{2-m} \| v \|_{H^2(\Omega)}, \qquad m=0,1
\end{equation}
For simplicity we use the notation $E v   = v$ and $\pi_h v = \pi_h E v$ when appropriate.

\item We shall also 
need an interpolation operator $P_h : L^2(\Omega) \rightarrow F_h W_{h,I}$, which we 
define by noting that the sets $S_h^{-1}(T)$ for $T \in \mcT_{h,I}$ provides a partition 
of $\mcT_h$. Then there is $\delta \sim h$ and a ball $B_{\delta,T}$ such that 
\begin{equation}
S_h^{-1}(T) \subset B_{\delta,T}
\end{equation}
On each ball $B_{\delta,T}$ there is $P_{h,T} v \in \mathbb{P}_1( B_{\delta,T} )$ such that 
\begin{equation}
\|\nabla^m ( v - P_{h,T} v)\|_{ B_{\delta,T}} \lesssim h^{2-m} \| v \|_{H^2(B_{\delta,T})}, \qquad m=0,1
\end{equation}
Defining $P_h:L^2(\mcTh) \rightarrow W_h$ by 
\begin{equation}
(P_h v)|_{S_h^{-1}(T)} =(P_{h,T} E v)|_{S_h^{-1}(T)} 
\end{equation}
we obtain the global error estimate
\begin{equation}\label{eq:Ph-error-est}
\|\nabla^m ( v - P_{h} v)\|_{\mcTh} \lesssim h^{2-m} \| v \|_{H^2(\Omega)},
\qquad m=0,1
\end{equation}
\bur{Observe also that $P_h$ satisfies $P_h v = F_h (P_h v)_I$, where we introduced the shorthand notation $(v)_I := v \vert_{\mcT_{h,I}}$}.
\item We define the interpolation operator $I_h: H^1(\Omega) \rightarrow 
V_h^E$ by \bur{ $I_h u := E_h
(\pi_h E u)_I$}.
\end{itemize}
\begin{lem} There is a constant such that for all $v \in H^2(\Omega)$,
\begin{equation}\label{eq:Ih-errorest}
\|E v - I_h v\|_{\mcTh} + h \|\nabla (E v -
I_h v)\|_{\mcTh} \lesssim h^2 |u|_{H^2(\Omega)}
\end{equation}
\end{lem}
\begin{proof} Adding and subtracting $\pi_h E v$ and $F_h (\pi_h E v )_I$ and using the 
triangle inequality 
\begin{align}
\| E v - I_h v \|_{H^m(\mcTh)} 
&= \| E v - E_h (\pi_h E v)_I \|_{H^m(\mcTh)}
\\
&= \| E v - \pi_h E v \|_{H^m(\mcTh)} 
+ \| \pi_h E v  - E_h (\pi_h E v)_I \|_{H^m(\mcTh)} 
\\
&= \| (I - \pi_h) E v  \|_{H^m(\mcTh)} 
+ \| \pi_h E v  - F_h (\pi_h E v)_I  \|_{H^m(\mcTh}  
\\
&\qquad 
+ \| (I-A_h) F_h (\pi_h E v)_I \|_{H^m(\mcTh)} 
\\
&=I + II + III
\end{align}
\paragraph{Term $\bfI$.} Using (\ref{eq:interpol}) we directly have
\begin{equation}
\| (I - \pi_h) E v  \|_{H^m(\mcTh)}  
\lesssim  h^{2-m} \| v \|_{H^2(\Omega)}
\end{equation}
\paragraph{Term $\bfI\bfI$.} Adding and and subtracting $P_h v$,
\bur{recalling} the identity $P_h v = F_h (P_h v)_I$, and using the triangle inequality 
we obtain
\begin{align}
& \| \pi_h E v  - F_h (\pi_h E v)_I  \|_{H^m(\mcTh)}   
\\
&\qquad  \leq
  \| \pi_h E v  -P_h v \|_{H^m(\mcTh)} +  \| P_h v  - F_h (\pi_h E v)_I \|_{H^m(\mcTh)}   
\\
&\qquad \leq 
  \| \pi_h E v  -P_h v  \|_{H^m(\mcTh)} +  \| F_h (P_h v  - \pi_h E v)_I  \|_{H^m(\mcTh)}    
\\
&\qquad \lesssim
  \|  \pi_h E v  -P_h v  \|_{H^m(\mcTh)} 
  \\
&\qquad \lesssim
 \|  \pi_h E v - v \|_{H^m(\mcTh)}  + \| v  -P_h v  \|_{H^m(\mcTh)} 
 \\
 &\qquad \lesssim
 h^{2-m}\|  v  \|_{H^2(\Omega)} 
\end{align}
where we used the stability estimates (\ref{eq:Fh-stab-L2}) for $m=0$ and (\ref{eq:Fh-stab-energy}) for $m=1$ for $F_h$, added and subtracted $v$ and used the triangle inequality, 
and used the interpolation error estimate (\ref{eq:interpol}) and (\ref{eq:Ph-error-est}).

\paragraph{Term $\bfI\bfI\bfI$.} Using the approximation result (\ref{eq:Ah-error-est}) for 
the average operator $A_h$, inserting the continuous function $\pi_h E v$ into the jump, and
using an inverse estimate to pass from faces to elements  we obtain
\begin{align}
 \| (I-A_h) F_h (\pi_h E v)_I \|_{H^m(\mcTh)} &\lesssim h^{-m} \| (I-A_h) F_h (\pi_h E v)_I \|_{\mcTh}
 \\
  &\lesssim h^{1/2-m} \|[F_h (\pi_h E v)_I] \|_{\mcFh}
   \\
  &\lesssim h^{1/2-m} \|[F_h (\pi_h E v)_I - \pi_h E v ] \|_{\mcFh}
  \\
   &\lesssim h^{-m} \|F_h (\pi_h E v)_I - \pi_h E v  \|_{\mcTh}
     \\
   &\lesssim h^{-m} \|F_h (\pi_h E v)_I - P_h v  \|_{\mcTh}
   +
    h^{-m} \|P_h v - \pi_h E v  \|_{\mcTh}
\end{align}
where we added and subtracted $P_h v$ and used the triangle inequality. The argument can 
now be concluded in the same way as for Term $II$. 
\end{proof}

\subsection{Finite Element Method}
In order to formulate the finite element method we use the following notations.
\begin{itemize}
\item Partition $[0,T]$ into $N$ intervals of length $k = T/N$ and 
let $t_n = n k$, for $n=0,1,\dots,N$. We let $u^n = u(t_n)$ and $v^n:\Omega \rightarrow\IR$ 
denotes a function  at time $t_n$. Define the discrete first (forward) and second (central) time differences
\begin{equation}
\partial_t v^{n} =\frac{v^{n+1}  -  v^n }{k}
\end{equation}
\begin{equation}
\partial_t^2 v^{n} = \frac{v^{n+1}  - 2 v^n + v^{n-1}}{k^2}
=\frac{1}{k} (\partial_t v^{n} - \partial_t v^{n-1})
\end{equation}
\item 
Define the central difference 
\begin{align}\label{def:diff-central}
\delta_t v^n = 
 \frac12 (\partial_t v^n + \partial_t v^{n-1})
\end{align}
and note for use below that we have the summation by parts formula 
\begin{align}\label{eq:sum-by-parts}
\sum_{n=1}^{N-1} 2k ( v^n, \delta_t w^n)
&=
(v^{N-1}, w^N) + (v^{N}, w^{N-1})
- (v^1, w^0) -  (v^0, w^1)
\\ 
&\qquad - \sum_{n=1}^{N-1} 2k ( \delta_t v^n, w^n)
\end{align}
\item For the spatial discretization we employ Nitsche's method and define 
the bilinear form 
\begin{equation}\label{eq:ah-def}
a_h(u,v) =(\nabla u,\nabla v) - (\nabla_n u, v)_{\partial \Omega}
- (u,\nabla_n v)_{\partial \Omega} + \gamma h^{-1} (u,v)_{\partial \Omega}
\end{equation}
where $\nabla_n = n \cdot \nabla$ with $n$ the exterior unit normal and $\gamma>0$ 
a parameter. 
\end{itemize}

\paragraph{Method.}
The cut finite element method takes the form:  for $n=1,\dots, N-1$, find 
$u_h^{n+1} \in V_h^E$, such that
\begin{equation}\label{eq:FEM}
(\partial_t^2 u^{n}_{h},v) + a_h(u_h^n, v) = (f^n,v), \qquad \forall v \in V_h^E
\end{equation}
with initial data $u^0_h, u^1_h \in V_h$ specified below.
%
The resulting updating formula takes the form
\begin{equation}
(u^{n+1}_h,v) = ( 2 u^n_h , v) - (u^{n-1}_h,v) + k^2 a_h(u^n_h,v) + k^2(f^n,v)
\end{equation}


\subsection{Matrix Formulation and Mass Lumping}

We formulate the method on matrix form and we replace the mass matrix with a diagonal 
matrix obtained by lumping the mass matrix in order to obtain an explicit method.
\begin{itemize}
\item Let $\{\varphi_i \}_{i \in \mcI_h}$, be the nodal basis in $V_h$
  enumerated by the index set $\mcI_h$, and let $\{\varphi_i \}_{i \in
    \mcI_{I,h}}$ be the nodal basis in $V_{h,I}$ \bur{ enumerated by
    the index set $\IIh$}.
Denote the dimensions of $V_h$ and $V_{h,I}$ by $N_h$ and $N_{h,I}$.  We then note that 
$\{ E_h \varphi_i \}_{i \in \mcI_{h,I}}$ is a basis in $V_h^E$.

\item Define the 
mass matrix, stiffness matrix, and load vector associated with the full finite element space $V_h$ by
\begin{equation}\label{eq:matrices_full}
(\hatM_{h} \hatv,\hatw)_{\mcI_{h}}= (v,w), \quad
(\hatA_h \hatv,\hatw)_{\mcI_{h}}= a_h(v,w), \quad
(\hatb_h, \hatw)_{\mcI_{h}} = (f,w)
\end{equation}
for all \bur{$v,w \in V_{h}$}. Here $\hatv$ denotes the coefficients of $v$ when expanded in the basis of $V_{h}$. 

\item Define the 
mass matrix, stiffness matrix, and load vector associated with the extended finite element space by
\begin{equation}\label{eq:matrices_Int}
(\hatM_{h,I} \hatv,\hatw)_{\mcI_{h,I}}= (v,w), \quad
(\hatA_{h,I} \hatv,\hatw)_{\mcI_{h,I}}= a_h(v,w), \quad
(\hatb_{h,I}, \hatw)_{\mcI_{h,I}} = (f,w)
\end{equation}
for all $v,w \in \bur{V_h^E}$. Here $\hatv$ denotes the coefficients of $v$ when expanded in the basis of $V_h^E$. 

\item Define the matrix representation of $E_h$ by
\begin{align}\label{eq:Eh-matrix}
(\hatE_h \hatv, \hatw)_{\mcI_h} = (\widehat{E_h v}, \hatw)_{\mcI_h}  
\end{align}
for all $v \in V_{h,I}$, $w \in V_h$. We note that $\hatE_h$ is an $N_h \times N_{h,I}$ 
matrix and that it 
follows from (\ref{eq:Eh-matrix}) that $\hatE_h \hatv = \widehat{E_h
  v}$. We then have \bur{for $v,w \in V_{h,I}$},
\begin{align}
(\hatv, \hatM_{h,I} \hatw)_{\IIh}
&=(E_h v, E_h w) 
= (\widehat{E_h v}, \hatM_h \widehat{E_h w} )_{\mcI_h} 
\\
&\qquad = (\hatE_h \hatv, \hatM_h \hatE_h \hatw)_{\IIh} 
= (\hatv, \hatE^T_h \hatM_h \hatE_h \hatw)_{\IIh} 
\end{align}
 Therefore the mass matrix on the extended finite element space can be expressed
in terms of the mass matrix on the full finite element space as follows
\begin{align}
\hatM_{h,I} = \hatE_h^T \hatM_h \hatE_h
\end{align}
and in the same way 
\begin{align}
\hatA_{h,I} =\hatE_h^T \hatA_h \hatE_h, \qquad \hatb_{h,I} = \hatE_h^T \hatb_h
\end{align}

\item  Define the lumped mass matrix 
$\hatM_L$ as the diagonal matrix with diagonal elements equal to the row sums 
of the mass matrix \bur{$\hatM_{h,I} $}, 
\begin{equation}\label{eq:lump-def}
\hatM_{L,{ij}} = 
\begin{cases}
0 & i \neq j
\\
\sum_{l \in \mcI_{h,I}(i) } \hatm_{il} & i = j
\end{cases}
\end{equation}
where for each $i \in \mcI_{h,I}$, 
\begin{equation}
\mcI_{h,I}( i ) = \{ j \in \mcI_{h,I} : \hatm_{ij} \neq 0 \}
\end{equation}
is the set of indices for which there is a nonzero entry in the $i$:th row (and column due to symmetry) of \bur{$\hatM_{h,I}$}. We also define the induced lumped mass inner product 
\begin{equation}
(v,w)_L = (\hatM_L \hatv, \hatw )_{\mcI_{h,I}} , \qquad v,w \in V_h^E
\end{equation}

\end{itemize}

\paragraph{Explicit Method.}
We define the lumped mass method:  for $n=1,\dots, N-1$, find 
$u^{n+1} \in V_h^E$, such that 
\begin{equation}\label{eq:FEM-lump}
(\partial_t^2 u^{n}_{h},v)_L + a_h(u_h^n, v) = (\bur{f^n_h},v)_L, \qquad \forall v \in V_h^E
\end{equation}
with initial data $u^0_h, u^1_h \in V_h$ \bur{ and $f^n_h \in V_h^E$ a
  suitable approximation of $f(t^n)$}. Using the fact that $\hatM_L$ is diagonal we obtain 
the explicit updating formula for $n=2,\dots, N-1$, 
\begin{equation}
\hatu^{n+1}_h = 2 \hatu^n_h - \hatu^{n-1}_h - k^2 \hatM^{-1}_L \hatA_{h,I}  \hatu^n_h 
+k^2 \hatb^n_L
\end{equation}
where $\hatb^n_L$ is the load vector associated with the lumped mass inner product
\begin{align}
(\bur{\hatb^n_L}, \hatv)_{\IIh} = (f^n,v)_L, \qquad v\in E_h \bur{V_{h,I}}
\end{align}
It follows that $\bur{\hatb^n_L} = \hatM_L \hatf^n$ where $\hatf^n$ is the 
internal nodal values of \bur{$f_h$}.
\section{Analysis of the Method}

\subsection{Ritz Projection}

\bur{In this section we will discuss the Ritz projection on the
  extended finite element space $V_h^E$. This will provided us
  with an interpolant with properties suitable for the error analysis
  of the wave equation. It also provides an analysis of Poisson's
  equation discretized using the $V_h^E$ in a cutFEM framework.}

Let 
\begin{equation}
\tn v \tn_h^2 = \| \nabla v \|^2 + h \| \nabla_n v \|^2_{\partial \Omega} + h^{-1} \| v \|^2_{\partial \Omega}
\end{equation}
\begin{lem}\label{lem:ah} The form $a_h$ defined in (\ref{eq:ah-def}) is continuous,
\begin{equation}\label{eq:ah-cont}
a_h(v,w) \lesssim \tn v \tn_h \tn w \tn_h, \qquad v,w \in H^{3/2 +\epsilon}(\Omega) + V_h
\end{equation}
and for $\gamma$ large enough coercive,
\begin{equation}\label{eq:ah-coer}
\tn v \tn_h^2 \lesssim a_h(v,v), \qquad v \in V_h^E
\end{equation}
\end{lem}
\begin{proof}
The continuity follows directly from Cauchy-Schwarz and  to establish the coercivity we 
start from
\begin{align}\label{eq:ah-a}
a_h(v,v) = \|\bur{\nabla} v \|^2 - 2 ( \nabla_n v, v)_{\partial \Omega} + \gamma h^{-1} \| v \|^2_{\partial \Omega} 
\end{align}
We have the estimate 
\begin{align}\label{eq:ah-b}
2( \nabla_n v, v)_{\partial \Omega} 
&\leq 2 \| \nabla_n v \|_{\partial \Omega}  \|v\|_{\partial \Omega} 
\\
&\leq C \| \nabla v \|^2_{\mcT_h(\partial \Omega)} h^{-1/2} \|v\|_{\partial \Omega} 
\\
&\leq  C^2 \delta  \| \nabla v \|^2_{\mcT_h(\partial \Omega)} + \delta^{-1} h^{-1} \|v\|^2_{\partial \Omega} 
\\
&\leq  C^2 \delta  \| \nabla v \|^2_{\mcT_{h,I}} + \delta^{-1} h^{-1} \|v\|^2_{\partial \Omega} 
\\ \label{eq:ah-c}
&\leq  C^2 \delta  \| \nabla v \|^2 + \delta^{-1} h^{-1} \|v\|^2_{\partial \Omega} 
\end{align}
where we used the inverse estimate $h^{1/2} \| \nabla v  \|_{\partial \Omega \cap T} \leq C \| \nabla v \|_T$, the stability (\ref{eq:Eh-stab}) of the discrete extension operator $E_h$, and finally the fact that $\mcT_{h,I} \subset \Omega$. Combining the estimates we find that 
\begin{align}
a_h(v,v) \geq (1 - C^2 \delta )\| \bur{\nabla} v \|^2 + (\gamma - \delta^{-1} ) h^{-1} \| v \|^2_{\partial \Omega} \gtrsim  \| \bur{\nabla} v \|^2 + h^{-1} \| v \|^2_{\partial \Omega} 
\end{align}
where we chose $\delta$ small enough and $\gamma$ large enough. Finally, (\ref{eq:ah-b})-(\ref{eq:ah-c}) give the estimate $h \| \nabla_n v \|^2_{\partial \Omega} \lesssim \| \nabla v \|^2$ the coercivity (\ref{eq:ah-coer}) follows.
\end{proof}

 In view of Lemma \ref{lem:ah}, we note that we can define the norm
\begin{equation}\label{eq:norm-ah}
\| v \|^2_{a_h} = a_h(v,v), \qquad v \in V_h^E
\end{equation}
directly associated with the Nitsche form, which is equivalent with 
$\tn \cdot \tn_h$ on $V_h^E$,
\begin{align}
 \| v \|_{a_h} \sim \tn v \tn_h,  \qquad v \in V_h^E
\end{align}
It will later be convenient to work with $\| \cdot \|_{a_h}$ instead of $\tn \cdot \tn_h$.

We begin by defining the Ritz projection $R_h : H^{s}(\Omega) \rightarrow E_h(V_{h,I})$, 
for $s>3/2$, by
\begin{equation}\label{eq:Rh-def}
a_h( R_h v , w ) = a_h (v, w) \qquad \forall w \in \bur{ V_h^E}
\end{equation}
\begin{lem} There is a constant such that,
\begin{equation}\label{eq:Rh-errorest}
\| \nabla^m (v - R_h v) \| \lesssim h^{2-m} \| v \|_{H^2(\Omega)},
\qquad m=0,1
\end{equation}
\end{lem}
\begin{proof}[Proof of (\ref{eq:Rh-errorest}).] Adding and subtracting an interpolant 
\begin{align}
\| \nabla^m ( v - R_h v ) \| &\leq \| \nabla^m ( v - I_h v ) \| + \|
                               \nabla^m ( I_h v - \bur{R_h} v ) \|
\\
&\lesssim h^{2-m} \| v \|_{H^2(\Omega)} + \| \nabla^m ( I_h v -
  \bur{R_h} v ) \|
\end{align}
where we used the interpolation error estimate (\ref{eq:Ih-errorest}) for the first term. For the second, coercivity (\ref{eq:ah-coer}), 
orthogonality (\ref{eq:Rh-def}),  and continuity (\ref{eq:ah-cont}),  give
\begin{align}
\tn I_h v - R_h v \tn^2_h &\lesssim a_h(I_h v - R_h v,\pi_h v - R_h v)
\\
&=a_h(I_h v - v,\pi_h v - R_h v)
\\
& \lesssim \tn I_h v - v\tn_h \tn \pi_h v - R_h v \tn_h
\end{align}
and therefore, using once more the interpolation estimate (\ref{eq:Ih-errorest}) for $I_h$,
\begin{align}
\tn I_h v - R_h v \tn_h &\lesssim \tn I_h v - v \tn_h 
\lesssim h \| v \|_{H^2(\Omega)}
\end{align}
The $L^2$ estimate is established using duality in the usual way.
\end{proof}

\begin{rem} Note that $u_h = R_h u$ is the finite element solution to 
\begin{align}
-\Delta u = f \quad \text{in $\Omega$}, 
\qquad u = 0 \quad \text{on $\partial \Omega$}
\end{align}
and thus (\ref{eq:Rh-errorest}) provides error estimates for a cut finite element method 
based on the extension operator $E_h$ for the Poisson equation.  
\end{rem}

\subsection{Estimate of the Lumping Error}
We begin by showing a stability estimate for the lumped inner product and then we prove 
an estimate of the consistency error resulting from lumping the mass matrix.  

Let $\| v \|^2_L = (v,v)_L$ be the norm associated with the lumped scalar product. We then have the stability 
\begin{equation}\label{eq:lump-stab}
\| v \|_{\mcTh} \lesssim \| v \|_L, \qquad v \in V_h^E
\end{equation}
This estimate follows from the $L^2$ stability (\ref{eq:Eh-stab}) of the extension operator followed by equivalence of the lumped product and the 
full $L^2$ product on the set of interior triangles
\begin{align}
 \| E_h v \|_{\mcTh} \lesssim \| v \|_{\mcT_{h,I}} \sim h^{d/2} \| \hatv \|_{\mcX_{h,I}} \sim \| v \|_L
\end{align} 
where $\mcX_{h,I}$ denotes the set of nodes in $\mcT_{h,I}$. \bur{Note
  that the last relation above holds since all elements of $M_{L}$
  must be $O(h^{d})$, since only interior nodes are considered.}

\begin{lem} There is  a constant such that 
\begin{equation}\label{eq:lump-est}
| (v,w) - (v,w)_L| \lesssim h^2 \| \nabla v \|_{\bur{\Omega}} \| \nabla w
\|_{\bur{\Omega}},\quad \bur{v, w \in V_h^E}
\end{equation}
\end{lem}
\begin{proof} Using the definitions \bur{(\ref{eq:matrices_Int})} and (\ref{eq:lump-def}) 
of the mass matrix $\hatM_{h,I}$ and the lumped mass matrix $\hatM_L$ we have 
\begin{align}\label{eq:lump-a}
 (v,w)_L - (v,w)  &= (\hatv, \hatM_L \hatw)_{\IIh} -  (\hatv, \bur{\hatM_{h,I}} \hatw)_{\IIh} 
\\
&\qquad =(\hatv, (\hatM_L - \bur{\hatM_{h,I}}) \hatw)_{\IIh}  = (\hatv, \hatB \hatw )_{\IIh}
\end{align}
with $\hatB = \hatM_L - \bur{\hatM_{h,I}}$. We note that $\hatB$ is indeed a graph Laplacian 
on the undirected weighted graph with vertices  $\mcX_{h,I}$, enumerated by $\mcT_{h,I}$, 
and edges 
\begin{equation}\label{eq:lump-b}
\mcE = \{ (i,j) : \hatB_{ij} \neq 0 \}
\end{equation}
with weights $\hatB_{ij}$. This follows from the fact that the diagonal elements of $\hatB$ 
is precisely the sum of the off diagonal elements in each row
\begin{align}
\hatB_{ii} = \sum_{\IIh(i) \setminus \{i\}} \hatB_{ij}
\end{align}
which is the sum of the weights on the graph edges that has node $i$ as a vertex. With each 
graph edge $E \in \mcE$ we associate the positive semi definite $N_{h,I} \times N_{h,I}$ matrix 
\begin{equation}\label{eq:lump-c}
B_E = 
\hatB_{ij} (e_i \otimes e_i  -  e_i \otimes e_j - e_j \otimes e_i + e_j \otimes e_j )
%
%
\end{equation}
where $\{e_i\}_{\IIh}$ is the canonical basis in $\IR^{N_{h,I}}$. Note
\bur{ that} $B_E$ maps the two dimensional space $\text{span}\{e_i,e_j\}$ into itself, and the corresponding matrix takes the form
\begin{equation}
{B}_E|_ {\text{span}\{e_i,e_j\}} =\hatB_{ij}  \left[
\begin{matrix}
1 & -1
\\
-1 & 1
\end{matrix}
\right]
\end{equation}
We then have 
\begin{equation}
\hatB = \sum_{E \in \mcE} B_E
\end{equation}
which gives
\begin{align}\label{eq:lump-d}
(\hatv, \hatB \hatw )_{\IIh} = \sum_{E \in \mcE} v_E B_{E} w_E =  \sum_{E \in \mcE} B_{ij} [v]_E [ w ]_E
\lesssim h^d\left( \sum_{E\in \mcE} [v]_E^2\right)^{1/2}\left( \sum_{E\in \mcE} [w]_E^2\right)^{1/2}
\end{align}
where $[v]_E = v_i - v_j$ is the difference between the nodal values $\{i,j\} = \IIh(E)$, connected by the edge $E$, and we used the bound $|\hatB_{ij}| \lesssim h^d$ which holds since $\hatB_{ij}$ is a bounded linear combination of elements in $\hatM_{h,I}$. Note 
that in the definition of $[\cdot]_E$ the order of $i$ and $j$ does not matter since we are working with a quadratic form with arguments that both are jumps. To estimate 
$\sum_{E\in \mcE} [v]_E^2$ we note that 
\begin{align}\label{eq:lump-e}
\sum_{E\in \mcE} [v]_E^2
\leq \sum_{i \in \IIh} \sum_{j \in \IIh(i)} (v_i - v_j)^2
\end{align}
where $\bur{\IIh(i)} \subset \IIh$ is the set of indices connected to
\bur{the node $i$} by an edge $E\in \mcE$.
Next let $\mcT_{h}(i)$ be the set of elements with at least one node in $\IIh(i)$ and note 
that it follows from the construction of the extension operator and shape regularity 
that there is a uniform bound, 
independent of $h \in (0,h_0]$ and $i\in \IIh$, on the number of elements in $\mcT_h(i)$ and that $\text{diam}(\mcT_{h}(i))\lesssim h$. We then have 
\begin{align}\label{eq:lump-f}
\sum_{j \in \IIh(i)} h^d \ (v_i - v_j)^2 \lesssim \| v_i - v \|^2_{\mcT_h(i)} 
\lesssim h^2 \| \nabla v \|^2_{\mcT_h(i)}
\end{align}
since $v \in V_h$. It follows that 
\begin{align}\label{eq:lump-g}
\sum_{E\in \mcE} [v]_E^2
\leq \sum_{i \in \IIh} \sum_{j \in \IIh(i)} (v_i - v_j)^2 \lesssim 
\sum_{i \in \IIh}  h^2 \| \nabla v \|^2_{\mcT_h(i)}
\lesssim h^2 \|\nabla v\|^2_{\mcT_h}
\end{align}
Combining  (\ref{eq:lump-d})  and (\ref{eq:lump-g}) \bur{and applying
  Lemma \ref{lem:stab_ext}}
we arrive at the desired estimate.
\end{proof}

\subsection{Discrete Stability}
\bur{
To prepare the terrain for the error analysis
we will prove stability for a slightly more general version of
(\ref{eq:FEM-lump}). Indeed we introduce a right hand side that
consists of two parts, expressed as functionals on $V_h$,  $r_1 =
\{r_1^n\}_{n=1}^{N}$ and $r_2 = \{r_2^n\}_{n=1}^{N}$, $r_i^n : V_h
\mapsto \mathbb{R}$. They will later be identified with two different
sources of approximation error driving the perturbation equation.  The
reason for this split is that optimal estimates require $r_1$ and
$r_2$ to be continuous with respect to different (discrete)
topologies, $r_1$ with respect to a discrete $H^1$-norm and $r_2$ with
respect to a discrete $L^2$-norm. This is a consequence of fact that the test function in
the derivation of the stability estimate is a discrete first order
time derivative and that the lumped mass approximation estimate
(\ref{eq:lump-est}) requires control of the gradient of the test
function. To avoid the appearance of mixed derivatives, that can not
be controlled, we apply
summation by parts in the $r_1$ part and move the discrete time
derivative from the test function to the functional. To provide bounds
in term of these functionals we recall the standard definition of norms
for linear functionals $l:V_h \rightarrow \IR$, using the appropriate norms,
\begin{equation}
\|  l \|_{a_h,\bigstar}  = \sup_{ v \in V_h^E\setminus \{ 0 \}} \frac{l(v)}{\| v \|_{a_h}}, 
\qquad 
\|  l \|_{L,\bigstar} = \sup_{ v \in V_h^E\setminus \{ 0 \}} \frac{l(v)}{\| v \|_L}
\end{equation}

The abstract scheme that we consider takes the form,} for $n=1,\dots, N-1$, find $v^{n+1} \in V_h^E$, such that
\begin{equation}\label{eq:FEM-lump-stab}
(\partial_t^2 v^{n},w)_L + a_h(v^n, w) = r^n(w), \qquad \forall w \in V_h^E
\end{equation}
given $v^0, v^1 \in V_h$. Here $r^n: V_h \rightarrow \IR$ are the linear functionals of 
the form 
\begin{equation}
r^n(v) = r^n_1 (v ) + r^n_2(v)
\end{equation}
\bur{Let us first introduce the continuities necessary for the two
  contributions $r_1$ and $r_2$, when their argument is a central difference of the form $k \delta_t v^n$.}
For $r_1(k \delta_t v^n)$, we sum over the contributions  $r^n_1$ and apply the summation by parts formula (\ref{eq:sum-by-parts}) to 
move the central difference from the test function of the form $k \delta_t v^n$ to the 
functional, 
\begin{align}
\sum_{n=1}^{N-1}  2k r^n_1(\delta_t v^n) 
&= r^{N-1}_1 (v^N) + r^{N}_1 (v^{N-1})
- r^1_1 (v^0) - r^0_1(v^1)
- \sum_{n=1}^{N-1} 2k (\delta_t r^n_1) (  v^n ) 
\\  \label{eq:r1-cont}
&\bur{\lesssim} \tn  r_1 \tn_{a_h,\bigstar} \max_{0\leq n \leq N} \|  v^n \|_{a_h}. 
\end{align}
Where we introduce the relevant norm of the functionals $\{r^n_1\}_{n=1}^{N}$,
\begin{align}\label{eq:norm-dual-ah}
\tn r_1 \tn_{a_h,\bigstar} &=  \|r_1^{N} \|_{a_h,\bigstar} + \|r_1^{N-1} \|_{a_h,\bigstar}  
+  \|r_1^{1} \|_{a_h,\bigstar} + \|r_1^{0} \|_{a_h,\bigstar} 
+ \sum_{n=0}^{N-1} k \|\partial_t r^n_1\|_{a_h,\bigstar}
\end{align}
Note that we used the identity 
(\ref{def:diff-central}) to pass from $\delta_t$ to
$\partial_t$. Next, for $r^n_2$ we only need continuity with
  respect to the $\| \cdot \|_L$ norm and therefore we do not need to
  move the time difference in this case
\begin{align}
r^n_2 (v) \leq \| r^n_2 \|_{L,\bigstar} \| v \|_L
\end{align}
which when acting on a function of the form $2k \delta_t v^n$ leads to the estimate 
\begin{align}\label{eq:r2-cont}
\sum_{n=1}^{N-1} 2k r^n_2 (\delta_t v^n)
\leq \sum_{n=1}^{N-1} 2k \|r^n_2 \|_{L,\bigstar} \| \delta_t^n v^n \|_L
\leq \tn r_2 \tn_{L,\bigstar} \max_{0 \leq n \leq N-1} \| \partial_t v^n \|_L 
\end{align}
where 
\begin{align}\label{eq:norm-dual-L}
\tn r_2 \tn_{L,\bigstar} = \sum_{n=1}^{N-1} 2k \|r^n_2 \|_{L,\bigstar}
\end{align}
Combining (\ref{eq:r1-cont}) 
and (\ref{eq:r2-cont}) we get 
\begin{align}\label{eq:r-cont}
\left| \sum_{n=1}^{N-1} 2 k r^n(\delta_t v^n) \right|
\lesssim 
\tn r_1 \tn_{a_h,\bigstar} \max_{0\leq n \leq N} \| v^n \|_{a_h} 
+ 
\tn r_2 \tn_{L,\bigstar} \max_{0 \leq n \leq N-1} \| \partial_t v^n \|_{L} 
\end{align}

\begin{lem} Let $v^{n+1}$, $n=1,\dots,N-1$, be defined by (\ref{eq:FEM-lump-stab}) and assume that (\ref{eq:r-cont}) is satisfied. If $k/h \leq c$ with $c$ sufficiently small. Then the following stability estimate holds
\begin{align} 
&\max_{\bur{2 \leq n \leq N}} \Big( \|\partial_t v^{n-1}\|_L^2 
+ \| v^n\|^2_{a_h} + \| v^{n-1}\|^2_{a_h} \Big)
\\ \label{eq:stab}
&\qquad
\lesssim 
\|\partial_t v^{0}\|_L^2  
+ \| v^1\|^2_{a_h} + \| v^0\|^2_{a_h} 
+ \tn r_1 \tn_{a_h,\bigstar}^2 
+ \tn r_2 \tn_{L,\bigstar}^2
\end{align}
\end{lem}
\begin{proof}
To prove stability we test \eqref{eq:FEM-lump-stab}  with $w = 
4k \delta_t v^n = 2k (\partial_t v^n 
+\partial_t v^{n-1})$ for $n=1,\dots,N-1$, and sum over the time levels, 
\begin{align}\label{eq:stab-a}
 \sum_{n=1}^{N-1} 2 k (\partial_t^2 v^{n}, \partial_t v^n 
+\partial_t v^{n-1})_L &+  \sum_{n=1}^{N-1} 2 k a_h ( v^n, \partial_t v^n 
+\partial_t v^{n-1})
\\
&\qquad =
 \sum_{n=1}^{N-1} 2 k r^n(\partial_t v^n 
+\partial_t v^{n-1})
\end{align}
Here the first term on the left hand side satisfies
\begin{align}\label{eq:stab-b}
\sum_{n=1}^{N-1} 2 k (\partial_t^2 v^{n}, \partial_t v^n 
+\partial_t v^{n-1})_L
= 2 \|\partial_t v^{N-1}\|_L^2
-  2 \|\partial_t v^{1}\|_L^2
\end{align}
since 
\begin{align}
k (\partial_t^2 v^n, \partial_t v^{n+1} + \partial_t v^n )_L
&=
(\partial_t v^n - \partial_t v^{n-1}, \partial_t v^{n} + \partial_t v^{n-1} )_L
\\ 
&=
\| \partial_t v^n \|^2_L - \| \partial_t v^{n-1} \|^2_L
\end{align}
Next for the second term we have 
\begin{align}
\sum_{n=1}^{N-1}  2k a_h(v^n, \partial_t v^n 
+\partial_t v^{n-1}   ) 
&=\sum_{n=1}^{N-1}  
2 a_h(v^n,v^{n+1} - v^{n-1}) 
\\ \label{eq:stab-c}
&=2 a_h( v^{N-1}, v^N) - 2 a_h( v^{1}, v^0) 
\end{align}
Inserting (\ref{eq:stab-b}) and (\ref{eq:stab-c}) into (\ref{eq:stab-a}) we obtain
\begin{align}\label{eq:stab-d}
2 \|\partial_t v^{N-1}\|_L^2 + 2 a_h( v^{N-1}, v^N)
&= 
2 \|\partial_t v^{0}\|_L^2+ 2 a_h( v^{0}, v^1)
\\ 
&\qquad 
+  \sum_{n=1}^{N-1} \bur{r^n(2 k (\partial_t v^n 
+\partial_t v^{n-1}))}
\end{align} 
Using the identities 
\begin{align}
k^2 \|\partial_t v^{N-1}\|_{a_h}^2 + 2 a_h( v^{N-1}, v^N) 
 &=
  \| v^N\|^2_{a_h} + \| v^{N-1}\|^2_{a_h}
  \\
  k^2 \|\partial_t v^{0}\|_{a_h}^2 + 2 a_h( v^{0}, v^1) 
 &=
  \| v^1\|^2_{a_h} + \| v^{0}\|^2_{a_h}
\end{align}
we may write (\ref{eq:stab-d}) in the form
\begin{align}
&2 \|\partial_t v^{N-1}\|_L^2 - k^2 \| \partial_t v^{N-1} \|_{a_h}^2
+ \| v^N\|^2_{a_h} + \| v^{N-1}\|^2_{a_h}
\\
&\qquad = 
2 \|\partial_t v^{0}\|_L^2   - k^2 \| \partial_t v^{0} \|_{a_h}^2
+ \| v^1\|^2_{a_h} + \| v^0\|^2_{a_h} 
+  \sum_{n=1}^{N-1} 2 k \bur{r^n}(\partial_t v^n 
+\partial_t v^{n-1}))
\end{align} 
Using an inverse inequality followed by the stability (\ref{eq:lump-stab}), we get
\begin{align}
\| w \|^2_{a_h} \lesssim h^{-2} \| w \|^2_{\mcTh} \lesssim h^{-2} \|w \|^2_L
\end{align}
which, with $w = \partial_t v^{N-1}$, gives
\begin{equation}\label{eq:stabation-d}
k^2 \|\partial_t v^{N-1}\|_{a_h}^2 \lesssim h^{-2} k^2 \|\partial_t v^{N-1}\|_{\mcTh}^2 \lesssim  h^{-2} k^2 \|\partial_t v^{N-1}\|_L
\end{equation}
Using the CFL condition $C h^{-2} k^2  \leq C c \leq 1$, 
where $C$ is the hidden constant in (\ref{eq:stabation-d}), and we may take $c$ 
small enough due to the assumption in the theorem, we arrive at
\begin{align}
&\|\partial_t v^{N-1}\|_L^2 
+ \| v^N\|^2_{a_h} + \| v^{N-1}\|^2_{a_h}
\\
& \leq 
2 \|\partial_t v^{0}\|_L^2  
+ \| v^1\|^2_{a_h} + \| v^0\|^2_{a_h} 
+  2\left|\sum_{n=1}^{N-1} 2 k \bur{r^n} (\delta_t v^n )\right|
\\ \label{eq:stab-f}
& \leq 
2 \|\partial_t v^{0}\|_L^2  
+ \| v^1\|^2_{a_h} + \| v^0\|^2_{a_h} 
\\
&\qquad 
+   2 \tn r_1 \tn_{a_h,\bigstar} \max_{\bur{0 \leq  n \leq N}} 
\| v^n \|_{a_h} +   2 \| r_2 \|_{L,\bigstar}  \max_{1 \leq  n \leq N-1}  \| \partial_t v^n \|_{L}
\end{align} 
where we used the identity $4 k \delta_t v^n = 2k ( \partial_t v^n + \partial_t v^{n-1} )$ and 
the bound (\ref{eq:r-cont}). Next keeping $N$ fixed on the right hand side, we note that (\ref{eq:stab-f})  holds with $N$ replaced by an arbitrary $n=2,\dots, N-1$ on 
the left hand side. Taking the maximum over $n$ on the left hand side we get 
\begin{align}
&\max_{\bur{2 \leq n \leq N}} \Big(  \|\partial_t v^{n-1}\|_L^2 
+ \| v^n\|^2_{a_h} + \| v^{n-1}\|^2_{a_h} \Big)
\leq 
2 \|\partial_t v^{0}\|_L^2  
+ \| v^1\|^2_{a_h} + \| v^0\|^2_{a_h} 
\\
&\qquad 
+   \tn r_1 \tn_{a_h,\bigstar} \max_{\bur{0 \leq  n \leq N}}  
\| v^n \|_{a_h} +   \| r_2 \|_{L,\bigstar}  \max_{1 \leq  n \leq N-1}  \| \partial_t v^n \|_{L}
\end{align} 
Finally, using a kick back argument we obtain
\bur{
 \begin{align}
&\frac12 \max_{2 \leq n \leq N-1} \Big( \|\partial_t v^{n-1}\|_L^2 
+ \| v^n\|^2_{a_h} + \| v^{n-1}\|^2_{a_h} \Big)
\leq 
2 (\|\partial_t v^{0}\|_L^2  
+ \| v^1\|^2_{a_h} + \| v^0\|^2_{a_h} )
\\
&\qquad 
+  \frac12 \tn r_1 \tn_{a_h,\bigstar}^2
 +   \frac12 \| r_2 \|_{L,\bigstar}^2
\end{align} 
}
which completes the proof.
\end{proof}

\subsection{Error Estimates}
\bur{
We will now combine the approximation properties and stability
estimates proved in the previous section to derive error estimates for
the cutFEM approximation. To simplify the notation we denote a
continuous function at a certain time level $t^n$, $v^n := v(t^n)$ and its
partial derivatives 
\begin{equation}
(d_t^m v)^n := \frac{\partial^m v}{\partial t^m}(t^n), \quad m \in \mathbb{N}_+
\end{equation}
for $m=1$ we will drop the superscript.

Before we derive the error estimates we recall the following
elementary results for the finite difference discretization in time.
}
\bur{
\begin{lem}\label{lem:Taylor}
For functions $v \in L^\infty(0,T;L^2(\Omega))$ there exists a positive constant
such that, if , $v^n:=v(t_n)$,
\begin{equation}
\|\partial_t^m v^n \|_L \lesssim \|d_t^m
v\|_{L^\infty(0,T;L^2(\Omega))}, \quad m\in \mathbb{N}_+
\end{equation}
and
\begin{equation}
\left(k \sum_{n=1}^{N-1} \|\partial_t^2 v^n - d_t^2 v^n\|^2\right)^{\frac12} \lesssim k^2 \|d_t^4 v\|_{L^2(0,T;L^2(\Omega))}
\end{equation}
\end{lem}}
\bur{
\begin{proof}
We only prove the first
inequality in the case $m=2$, the cases $m=1$ and $m=3$ are similar. Using partial
integration we see that
\begin{align}
\partial_t^2 v^n &= \frac{1}{k^2} (v^{n+1} - 2 v^n + v^{n-1}) 
\\
&=
\frac{1}{k^2} \left(\int_{t^{n}}^{t^{n-1}}(t^{n-1}-t)\frac{\partial^2 u}{\partial
  t^2}(t) ~\mbox{d}t + \int_{t^{n}}^{t^{n+1}} (t^{n+1}-t) \frac{\partial^2 u}{\partial
  t^2}(t) ~\mbox{d}t\right)
\\
&\leq 2 \|d_t^m
v\|_{L^\infty(t^{n-1},t^{n};L^2(\Omega))}
\end{align}
Once again by partial integration it follows that
\begin{equation}
\partial_t^2 v^n - d_t^2 v^n  = \frac{1}{k^2} \left(\int_{t^{n}}^{t^{n-1}}\frac{(t^{n-1}-t)^3}{6}\frac{\partial^4 u}{\partial
  t^2}(t) ~\mbox{d}t + \int_{t^{n}}^{t^{n+1}} \frac{(t^{n+1}-t)^3}{6} \frac{\partial^4 u}{\partial
  t^2}(t) ~\mbox{d}t\right)
\end{equation}
Using Cauchy-Schwarz inequality in the right hand side we have
\begin{equation}
k^{-2} \int_{t^{n}}^{t^{n+1}} \frac{(t^{n+1}-t)^3}{6} \frac{\partial^4 u}{\partial
  t^2}(t) ~\mbox{d}t \leq \frac16 k^{\frac32} \|d_t^4 u\|_{{L^2(t^n,t^{n+1})}
  t^2}.
\end{equation}
Therefore
\begin{equation}
 \|\partial_t^2 v^n - d_t^2 v^n\|^2\lesssim k^{3} \|d_t^4
 u\|_{{L^2(t^n,t^{n+1}; L^2(\Omega))}}^2
\end{equation}
The claim then follows by summing over $n$, multiplying by $k$ and taking
square roots of both sides.
\end{proof}
}

\begin{thm} Let $u^{n+1}_h$, for $n=1,\dots, N-1,$ be defined by
  (\ref{eq:FEM-lump}) with initial data $u^0_h = R_h u^0$ and $u^1_h =
  R_h u^1$. Then \bur{if $u$ is a sufficiently smooth solution to \eqref{eq:wave}}, the following error estimates hold 
\begin{align}\label{eq:err-est-L2}
\bur{\|(d_t u)^{N-1} - \partial_t (u_h^{N-1})\| + \|u^{N} - u_h^{N}\|
+ \|u^{N-1} - u_h^{N-1} \|}
\lesssim h^2 + k^2
\end{align}
\begin{align}\label{eq:err-est-energy}
 \bur{\|\nabla (u^{N} - u_h^{N})\|+ \|\nabla (u^{N-1} - u_h^{N-1} )\|}
\lesssim h + k^2
\end{align}
\end{thm}

\begin{proof}
We first note that the exact solution satisfies 
\begin{equation}\label{eq:err-a}
((d_t^2 u)^n , v) + a_h(u^n,v ) =  (\bur{f^n},v^n) \qquad \forall v\in V_h, t \in (0,T)
\end{equation}
and for $n=1,\dots,N-1$, the numerical scheme satisfies
\begin{equation}\label{eq:err-b}
(\partial_t^2 u^n_h , v)_L + a_h(u^n_h,v ) = (\bur{f_h^n},v^n)_L \qquad \forall v \in \bur{V_h^E}
\end{equation}
Subtracting the two equations we obtain the error equation
\begin{align}\label{eq:err-c}
 ( (d_t^2 u)^n , v) - (\partial_t^2  u^n_h, v)_L  + a_h(u^n - u^n_h,v)  
 = \bur{(f^n,v) - (f_h^n,v)_L}
  \qquad \forall v \in \bur{V_h^E}
\end{align}
In order to estimate the error we split it into two contributions using the Ritz projection,
\begin{equation}\label{eq:err-d}
u^n - u_h^n = u^n - R_h u^n + R_h u^n - u_h^n = \rho^n + \theta^n
\end{equation}
\bur{In the standard manner we then split the norms in the left hand
  side of \eqref{eq:err-est-L2} and \eqref{eq:err-est-energy} using
  the triangle inequality in the contributions from $\rho^n$ and
  $\theta^n$, $\|u^n - u_h^n\| \leq \| \rho^n\| + \|\theta^n\|$.}
In the following paragraphs we estimate the two contribution to the error emanating from the interpolation error $\rho$ and the discrete part of the error $\theta$. The $\rho$ contribution can be directly estimated using the error estimates (\ref{eq:Rh-errorest}) 
for the Ritz projection. For the $\theta$ contribution we derive an error equation with 
a right hand side that acccounts for the lumping error and the error in the difference approximation of the second order time derivative. The bound for $\theta$ is then obtained 
by applying the stability estimate (\ref{eq:stab}) followed by a priori bounds for the right 
hand side.

\paragraph{The $\boldsymbol{\rho}$ Contribution.}
Applying the error estimate (\ref{eq:Rh-errorest})  for the Ritz projection we have the estimates
\begin{align}\label{eq:err-e}
\| \rho^n \| &\lesssim h^2 \|u^n \|_{H^2(\Omega)}
\\ \label{eq:err-f}
\| \bur{ (d^m_t \rho)^n }\| &\lesssim h^2 \|(\bur{d^m_t u})^n
                            \|_{H^2(\Omega)}, \quad \bur{m=1,2,3}
\\ \label{eq:err-g}
\| \rho^n \|_{a_h}  &\lesssim h \|u^n \|_{H^2(\Omega)}
\end{align} 
where we used the commutation $\bur{(d_t R_h v)^n = R_h (d_t v)^n}$.

\paragraph{The $\boldsymbol{\theta}$ Contribution.}
We note that we have the identity 
\begin{align}\label{eq:err-h}
 ( (d_t^2 u)^n , v) - (\partial_t^2  u^n_h, v)_L  
&=  
 ( (d_t^2 u)^n , v) 
 -(\partial_t^2 (R_h u^n), v)_L  
+ (\partial_t^2 (R_h u - u_h)^n, v)_L  
\\ \label{eq:err-i}
&=  ( (d_t^2 u)^n , v) 
 - (\partial_t^2 (R_h u^n), v)_L   + (\partial_t^2 \theta^n,v)_L
\end{align}
 and using the orthogonality of $R_h$,
\begin{align}\label{eq:err-k}
a_h(u^n - u^n_h,v)  = a_h^n ( \rho, v) + a_h(\theta,v) = a_h(\theta, v)
\end{align}
Combining (\ref{eq:err-h}), (\ref{eq:err-i}), and (\ref{eq:err-k}), we get 
the following error equation for the discrete part $\theta$ of the error
\begin{align}\label{eq:err-l}
(\partial_t^2  \theta^n, v)_L  + a_h(\theta^n,v)  
=\underbrace{ (f^n,v) - (f^n,v)_L+  (\partial_t^2 (R_h u^n), v)_L  
-( (d_t^2 u)^n , v)}_{ r^n(v) }
\end{align}
where we introduced the functional $r^n: V_h \rightarrow \IR$. We now split $r^n$, by 
adding and subtracting suitable term, in order to apply a stability bound of the form 
(\ref{eq:r-cont}), 
\begin{align}\label{eq:err-m}
r^n(v)
&=   (f^n,v) - (f^n_h,v)_L+
(\partial_t^2 R_h u^n , v)_L -  ( (\partial_t^2 u)^n , v)
\\
&=  \underbrace{(f^n,v)- (f^n_h,v)_L + (\partial_t^2 R_h u^n , v)_L -  (\partial_t^2 R_h u^n , v) }_{r^n_1(v)}
\\
&\qquad \qquad +
\underbrace{ (\partial_t^2 R_h u^n , v)  -  ((d_t^2 R_h u)^n , v) 
+
((d_t^2 R_h u)^n , v)   - ( (d_t^2 u)^n , v)}_{r^n_2(v)}
\\
&=r_1^n(v) + r_2^n(v)
\end{align} 
where we have collected the terms associated with the lumping error in $r_1$ and the 
remaining terms in $r_2$. Below will prove
the following bounds on the residuals $r_1$ and $r_2$. 
\bur{
\begin{align}
\label{eq:err-p}
\tn r_{1} \tn_{a_h,\bigstar}  &\lesssim  h^2(\|u\|_{W^{3,\infty}(0,T;H^1(\Omega))}+\|f\|_{W^{1,\infty}(0,T;H^2(\Omega))})
\\[3mm]
\label{eq:err-o}
\tn r_2 \|_{L,\bigstar}  &\lesssim 
 k^2 \|d_t^4 u\|_{L^2(0,T;L^2(\Omega)} + h^2 \|u\|_{W^{2,\infty}(0,T;H^2(\Omega))}
\end{align}
Here we have
omitted higher order terms.
}
\bur{Anticipating the approximation error estimates \eqref{eq:err-p} and \eqref{eq:err-o}
we may use the stability estimate (\ref{eq:stab}),} where $\theta^0 = \theta^1 =0$ since 
$u^0_h = R_h u^0$ and $u^1_h = R_h u^1$, to obtain 

\begin{align}
&\max_{1 \leq n \leq N-1} \Big( \|\partial_t \theta^{n-1}\|_L^2 
+ \| \theta^n\|^2_{a_h} + \| \theta^{n-1}\|^2_{a_h} \Big) \lesssim
\tn r_1 \tn^2_{a_h,\bigstar} +   \| r_2 \|^2_{L,\bigstar}  
\\
&\qquad \lesssim
h^4 \left( \|u\|_{W^{3,\infty}(0,T;H^1(\Omega))}+\|f\|_{W^{1,\infty}(0,T;H^2(\Omega))}\right)^2
\\
&
\qquad\qquad + \left( \sum_{n=1}^{N-1} k h^2 \| \partial_t^2 u^n\|_{H^2(\Omega)} 
+ k^3 \| \partial_t^4 u \|_{L^\infty(L^2(\Omega))} \right)^2
\\
&\qquad \lesssim (h^2 + k^2)^2
\end{align}

\paragraph{Verification of (\ref{eq:err-p}).} Starting from the definition 
(\ref{eq:norm-dual-ah}),
\begin{align}\label{eq:rep_dual}
\tn r_1 \tn_{a_h,\bigstar} &=  \|r_1^{N} \|_{a_h,\bigstar} + \|r_1^{N-1} \|_{a_h,\bigstar}  
+  \|r_1^{1} \|_{a_h,\bigstar} + \|r_1^{0} \|_{a_h,\bigstar} 
+ \sum_{n=0}^{N-1} k \|\partial_t r^n_1\|_{a_h,\bigstar}
\end{align}
with 
\begin{align}
r_1^n(v)&=
\underbrace{(f^n,v)- (f^n_h,v)_L}_{I} +\underbrace{ (\partial_t^2 R_h u^n , v)_L -  (\partial_t^2 R_h u^n , v)}_{II}
\end{align}
\bur{We start with estimates of the first four terms in the right hand side of
  (\ref{eq:norm-dual-ah}), by considering an arbitrary $n$. By adding
  and subtracting $(f^n_h,v)$ we have
\begin{equation}
I =  (f^n,v) - (f^n_h,v)_L = (f^n,v) - (f^n_h,v)+(f^n_h,v) - (f^n_h,v)_L.
\end{equation}
Assuming that $f^n_h$ has optimal approximation properties we see that
\begin{equation}
 (f^n,v) - (f^n_h,v) \lesssim h^2 \|f^n\|_{H^2(\Omega)} \| v \|_{a_h}
\end{equation}
where we used the Poincar\'e inequality $\|v\| \lesssim \| v
\|_{a_h}$.
For the second term and term $II$ we apply Lemma \ref{eq:lump-est} to obtain
\begin{equation}
(f^n_h,v) - (f^n_h,v)_L \lesssim h^2 \|\nabla f_h^n\| \| v
\|_{a_h} \lesssim h^2 ( \|\nabla f^n\| + h\|f^n\|_{H^2(\Omega)}) \| v \|_{a_h}
\end{equation}
and
\begin{equation}
(\partial_t^2 R_h u^n,v) - (\partial_t^2 R_h u^n,v)_L \lesssim h^2 \|\nabla \partial_t^2 R_h u^n\| \| v \|_{a_h}
\end{equation}
Applying the first inequality of Lemma \ref{lem:Taylor}, adding and subtracting $\nabla d_t^2
u^n$ and applying approximation shows that 
\begin{equation}
 \|\nabla \partial_t^2 R_h u^n\| \leq \|\nabla d_t^2 \rho^n\| +
 \|\nabla d_t^2 u^n\| \lesssim \|\nabla d_t^2 u^n\|   + h \|d_t^2
 u^n\|_{H^2(\Omega)}
\end{equation}
To sum up we have (neglecting  higher order terms)
\begin{align}
&\|r_1^{N} \|_{a_h,\bigstar} + \|r_1^{N-1} \|_{a_h,\bigstar}  
+  \|r_1^{1} \|_{a_h,\bigstar} + \|r_1^{0} \|_{a_h,\bigstar}  
\\
&\qquad \lesssim h^2(\|u\|_{W^{2,\infty}(0,T;H^1(\Omega))}+\|f\|_{L^{\infty}(0,T;H^2(\Omega))})
\end{align}
To control the last term in the right hand side of
\eqref{eq:rep_dual}, we simply apply the above arguments to
$\partial_t f^n$, $\partial_t f_h^n$ and $\partial_t \partial_t^2 R_h
u^n$. This results in similar bounds, but with an additional time
derivative.
\begin{align}
 \sum_{n=0}^{N-1} k \|\partial_t r^n_1\|_{a_h,\bigstar} 
 &\lesssim k \sum_{n=0}^{N-1}
 h^2(\|u\|_{W^{3,\infty}(t^{n},t^{n+1};H^1(\Omega))}+\|f\|_{W^{1,\infty}(t^{n},t^{n+1};H^2(\Omega))}) 
 \\
& \lesssim h^2(\|u\|_{W^{3,\infty}(0,T;H^1(\Omega))}+\|f\|_{W^{1,\infty}(0,T;H^2(\Omega))})
\end{align}

\paragraph{Verification of (\ref{eq:err-o}).} 
We recall the definition \eqref{eq:norm-dual-L}
\begin{align}\label{{eq:norm-dual-L_rep}}
\tn r_2 \tn_{L,\bigstar} = \sum_{n=1}^{N-1} 2k \|r^n_2 \|_{L,\bigstar}
\end{align}
Each $r^n_2$ in the right hand side can be bounded as follows. Using
the stability of $v_h \in V_h^E$ we see that for all $w\in L^2(\Omega)$,
\begin{equation}
(w,v_h) \leq \|w\| \|v_h\| \lesssim \|w\| \|v_h\|_L
\end{equation}
in particular
\begin{align}
&(\partial_t^2 R_h u^n- (d_t^2 R_h u)^n , v) 
+
((d_t^2 R_h u)^n - (d_t^2 u)^n , v) 
\\
&\qquad
\lesssim (\|\partial_t^2 R_h u^n -  (d_t^2 R_h u)^n\|+\|
((d_t^2 R_h u)^n - (d_t^2 u)^n\|) \|v_h\|_L
\end{align}
By
the definition of $r^n_2$ we then have
\begin{equation}
\tn r_2 \tn_{L,\bigstar}=  \underbrace{2k \sum_{n=1}^{N-1}
  \|\partial_t^2 R_h u^n -  (d_t^2 R_h u)^n\|}_{I}
+\underbrace{2k \sum_{n=1}^{N-1}\|(d_t^2 \rho)^n\|}_{II}
\end{equation}
The term $I$ is bounded using the second inequality of Lemma
\ref{lem:Taylor} and then, since we have not proved $L^2$-stability of
$R_h$, we add and subtract $d_t^4 u$, use the triangle inequality and
the inequality \eqref{eq:err-f}
\begin{align}
I &\lesssim k^2 \|d_t^4 R_h u\|_{L^2(0,T;L^2(\Omega)} 
\\
&\lesssim k^2
(\|d_t^4 u\|_{L^2(0,T;L^2(\Omega)}+\|d_t^4 (u - R_h
u)\|_{L^2(0,T;L^2(\Omega)})
\\
&\lesssim  k^2
(\|d_t^4 u\|_{L^2(0,T;L^2(\Omega)}+h^2 \|d_t^4 u\|_{L^2(0,T;H^2(\Omega)})
\end{align}
For $II$ we apply \eqref{eq:err-f} and take the max over the
time levels to obtain
\begin{equation}
2k \sum_{n=1}^{N-1}\|(d_t^2 \rho)^n\| \lesssim h^2 \|u\|_{W^{2,\infty}(0,T;H^2(\Omega))}
\end{equation}
We conclude that, omitting high order terms we have, as claimed,
\begin{equation}
\tn r_2 \tn_{L,\bigstar} \lesssim k^2 \|d_t^4 u\|_{L^2(0,T;L^2(\Omega)} + h^2 \|u\|_{W^{2,\infty}(0,T;H^2(\Omega))}
\end{equation}
}
\end{proof}

\section{Numerical Examples}

In the numerical examples below, we use the following implementation of the extension operator. The mapping $S_h$ is constructed by associating with each element $T\in \mcT_h \setminus \mcT_{h,I}$ the element $S$ in $\mcT_{h,I}$ which minimizes the distance 
between the element centroids. For each $x \in \mcX_h \setminus \mcX_{h,I}$ the weights in the nodal average 
$\langle \cdot \rangle_x$, see (\ref{eq:average-nodal}), is taken to be $1$ on precisely one element $T_x \in \mcT_h(x)$ and 
zero on all elements in $\mcT_h(x) \setminus T_x$, where we recall that $\mcT_h(x)$ is the set of elements which has $x$ as a vertex. Note that this choice of weights corresponds to simply 
defining the nodal value in $x \in \mcX_h \setminus \mcX_{h,I}$ by $((F_h v)|_{T_x})|_x$, where 
$F_h$ is defined in (\ref{eq:Fh}). This particular implementation has the advantage that it introduces relatively few non zero elements in the mass and stiffness matrix. The Nitsche parameter was set to $\gamma =10$ in all computations and the initial data is the extension of nodal interpolant in interior nodes. 

\subsection{Space-Time Convergence}

On the disc $\Omega = \{ r: \; r < 0.5\}$, $r=\sqrt{x^2 +y^2}$, we consider a problem with manufactured solution
\begin{equation}
u =(1-4r^2)\cos{(\omega t)}
\end{equation}
corresponding to the right hand side
\begin{equation}
f=(4\omega^2 r^2 - \omega^2 + 16)\cos{(\omega t)}
\end{equation}
with $\omega=2\pi$. We solve this problem over one period, i.e., with $T=1$. The timestep $k$ is coupled to the meshsize $h$ by $k\sim h$. On our inital mesh $h=2.69\times 10^{-2}$ and $k=\pi/2000 \approx 1.57\times 10^{-3}$. 

In Figure \ref{fig:elevoneperiod} we show the solution (on the third mesh in a sequence of halving the meshsize) after one period, and in Figure \ref{fig:convtimespace}
we show the convergence at time $T$ in $L_2(\Omega)$ and in $H^1(\Omega)$. The expected convergence of $O(h^2)$ is attained in $L_2$ and $O(h)$ in $H^1$.

\subsection{Dirichlet vs. Neumann}

In this example we show the effect of a pulse approaching the boundary for zero Dirichlet boundary conditions and for zero Neumann boundary conditions.
The domain is the same as in the previous example, we set $h=7\times 10^{-3}$, $k=3.93\times 10^{-4}$. The initial solution is given by 
\begin{equation}
u(r,0) =1+\cos(\pi r/r_0)\;\;\text{if $ r< r_0$}, u(r,0)=0\;\;\text{elsewhere, and}\; \partial_t u =0
\end{equation}
with $r_0 = 0.2$. An interpolated initial condition on the computational mesh is shown in Figure \ref{fig:initial}. In Figure \ref{fig:solD}
we show the Dirichlet solution after $t=0.35$ and $t=0.4$, and in Figure \ref{fig:solN} we show the Neumann solution at the same times.
The method can clearly handle both hard and soft boundary conditions without modification.

\subsection{Increasing Frequency}

Here we show the effect of a pulse with decreasing support approaching the boundary. 
Our domain is $(-0.81,0.79)\times(-0.8,0.8)$ and has Neumann boundary conditions on the uncut boundaries $y=\pm 0.8$. On the uncut boundary $x=-0.81$ we impose Dirichlet conditions strongly, and on the cut boundary at $x=0.79$ we impose zero Dirichlet boundary conditions weakly. In Fig. \ref{fig:cut} we show how the mesh is cut in a closeup.
We set $h=8.9\times 10^{-3}$, $k=3.93\times 10^{-4}$. The initial solution is given by 
\begin{equation}
u(x,y) =(1+\cos(\pi \vert x+0.01\vert/d_0)\;\;\text{if $ \vert x+0.01\vert< d_0$}, u(x,y)=0\;\;\text{elsewhere}
\end{equation}
and $\partial_t u =0$,
with different $d_0$. This pulse splits into two, one going left and hitting the uncut boundary, one going right and hitting the cut boundary.
We show snapshots of the solutions different times and for different $d_0$ in Figs. \ref{fig:pulsestart}--\ref{fig:pulseend}. Note the dispersion error
which becomes more pronounced as $d_0$ decreases. The difference in quality of the solution at the uncut and cut boundaries boundary is small
and does not become more pronounced as the support of the pulse decreases. We note that as the frequency increases, the meshsize must (eventually) be
decreased to avoid dispersion errors, which means the weak Dirichlet data will also be resolved better. 

\bigskip
\paragraph{Acknowledgements.}
This research was supported in part by the Swedish Foundation
for Strategic Research Grant No.\ AM13-0029, the Swedish Research
Council Grants Nos.\  2013-4708, 2017-03911, and the Swedish
Research Programme Essence. EB was supported in part by the EPSRC grant EP/P01576X/1.

\bibliographystyle{abbrv}
\footnotesize{
\bibliography{ref}
}

\bigskip
\bigskip
\noindent
\footnotesize {\bf Authors' addresses:}

\smallskip
\noindent
Erik Burman,  \quad \hfill \addressuclshort\\
{\tt e.burman@ucl.ac.uk}

\smallskip
\noindent
Peter Hansbo,  \quad \hfill \addressjushort\\
{\tt peter.hansbo@ju.se}

\smallskip
\noindent
Mats G. Larson,  \quad \hfill \addressumushort\\
{\tt mats.larson@umu.se}

\newpage

\begin{figure}[ht]
	\begin{center}
		\includegraphics[scale=0.20]{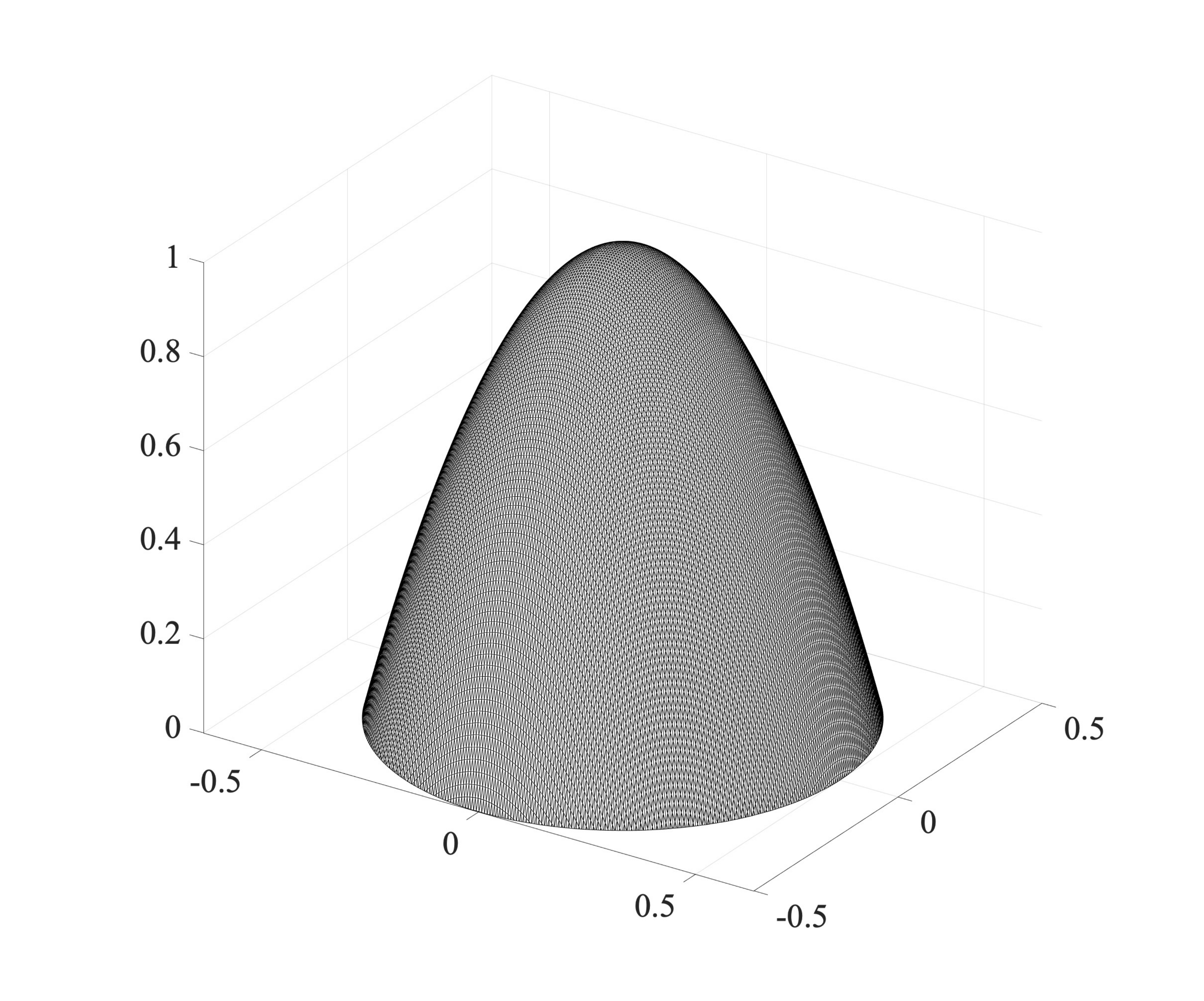}
	\end{center}
	\caption{Elevation of the computed solution on a particular mesh.}
\label{fig:elevoneperiod}
\end{figure}

\begin{figure}[ht]
	\begin{center}
		\includegraphics[scale=0.20]{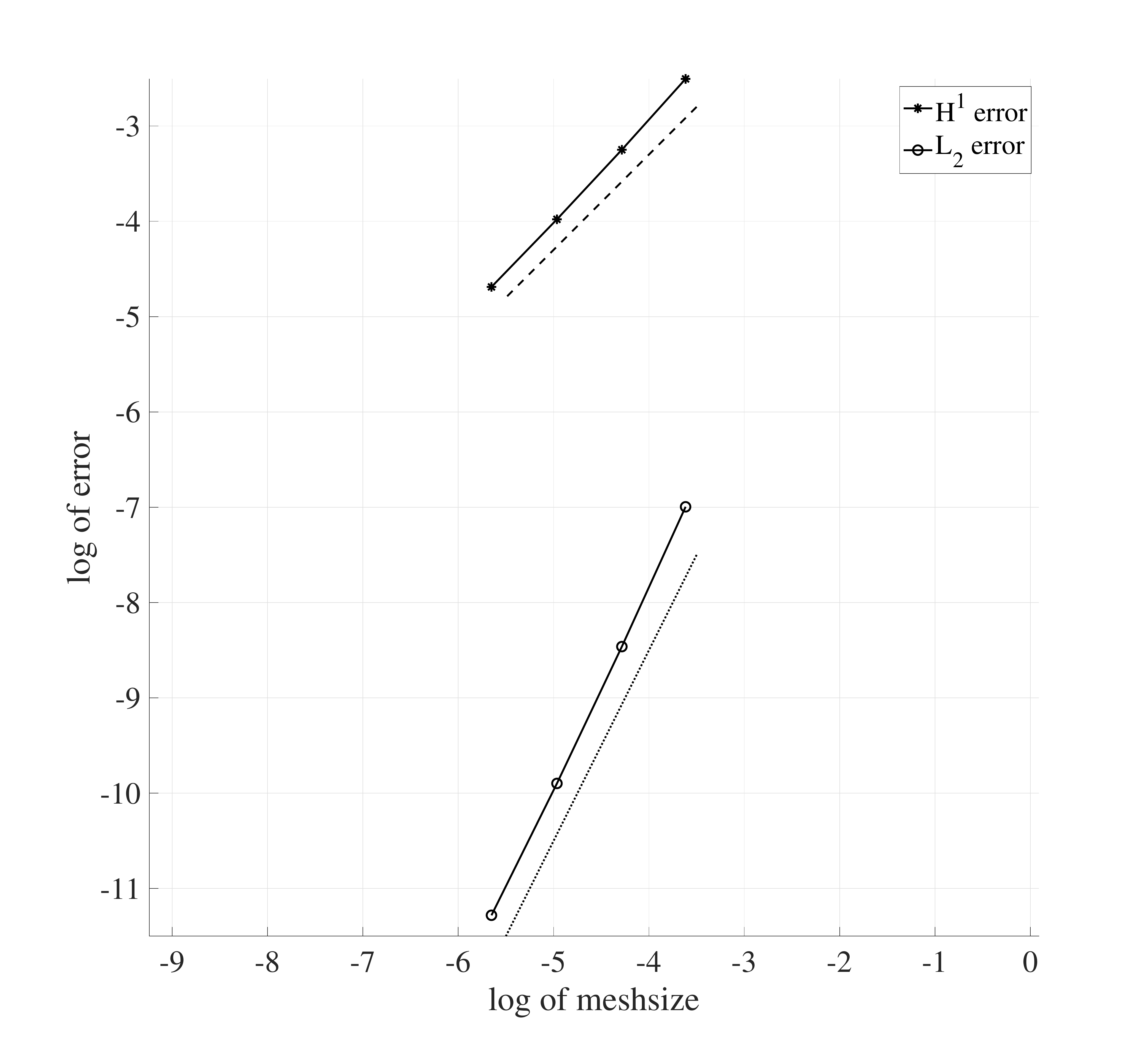}
	\end{center}
	\caption{Convergence at time $T=1$. Dashed line has inclination 1:1, dotted line has inclination $2:1$.}
	\label{fig:convtimespace}
\end{figure}

\begin{figure}[ht]
	\begin{center}
		\includegraphics[scale=0.25]{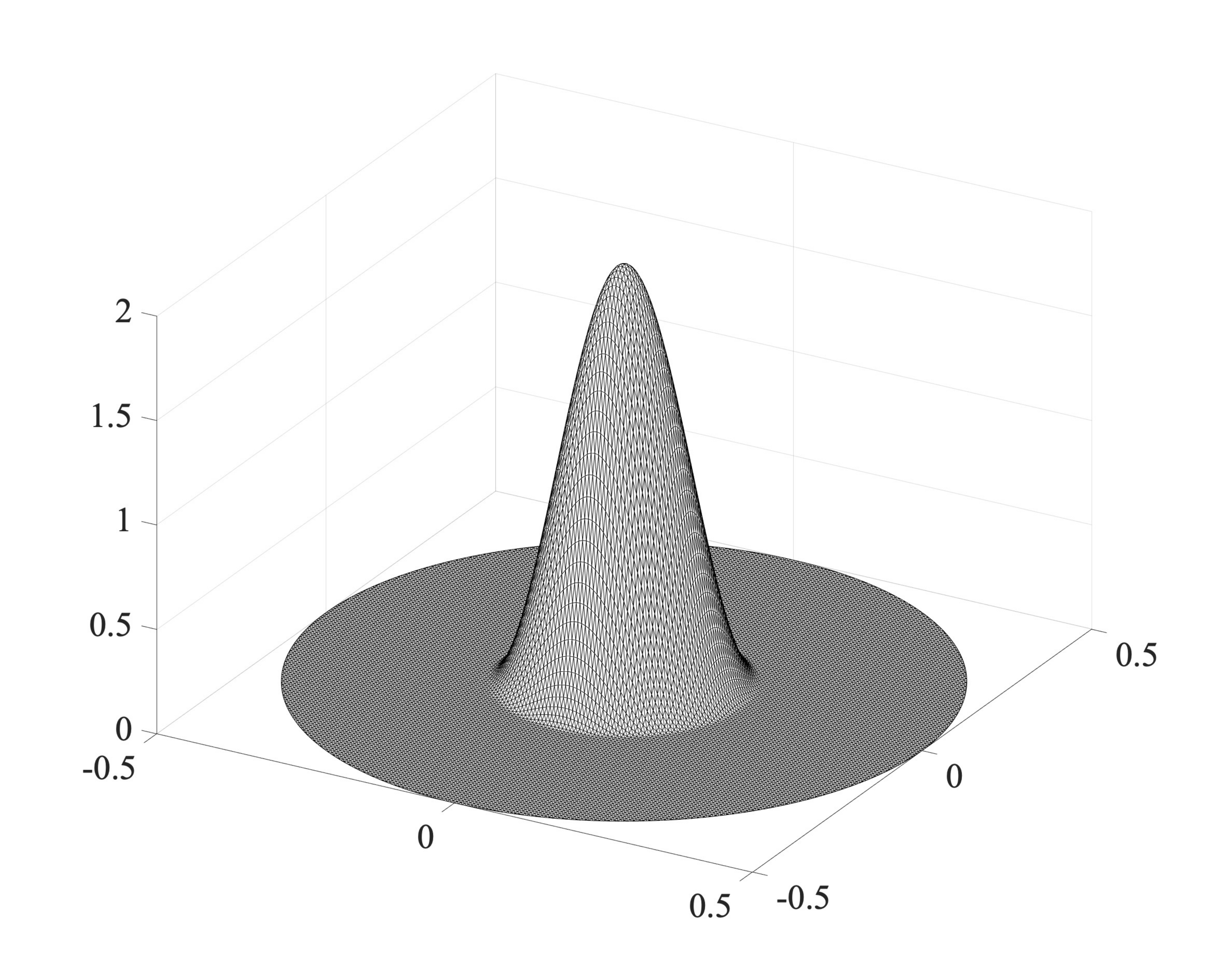}
	\end{center}
	\caption{Initial pulse.}
	\label{fig:initial}
\end{figure}

\begin{figure}[ht]
	\begin{center}
		\includegraphics[scale=0.18]{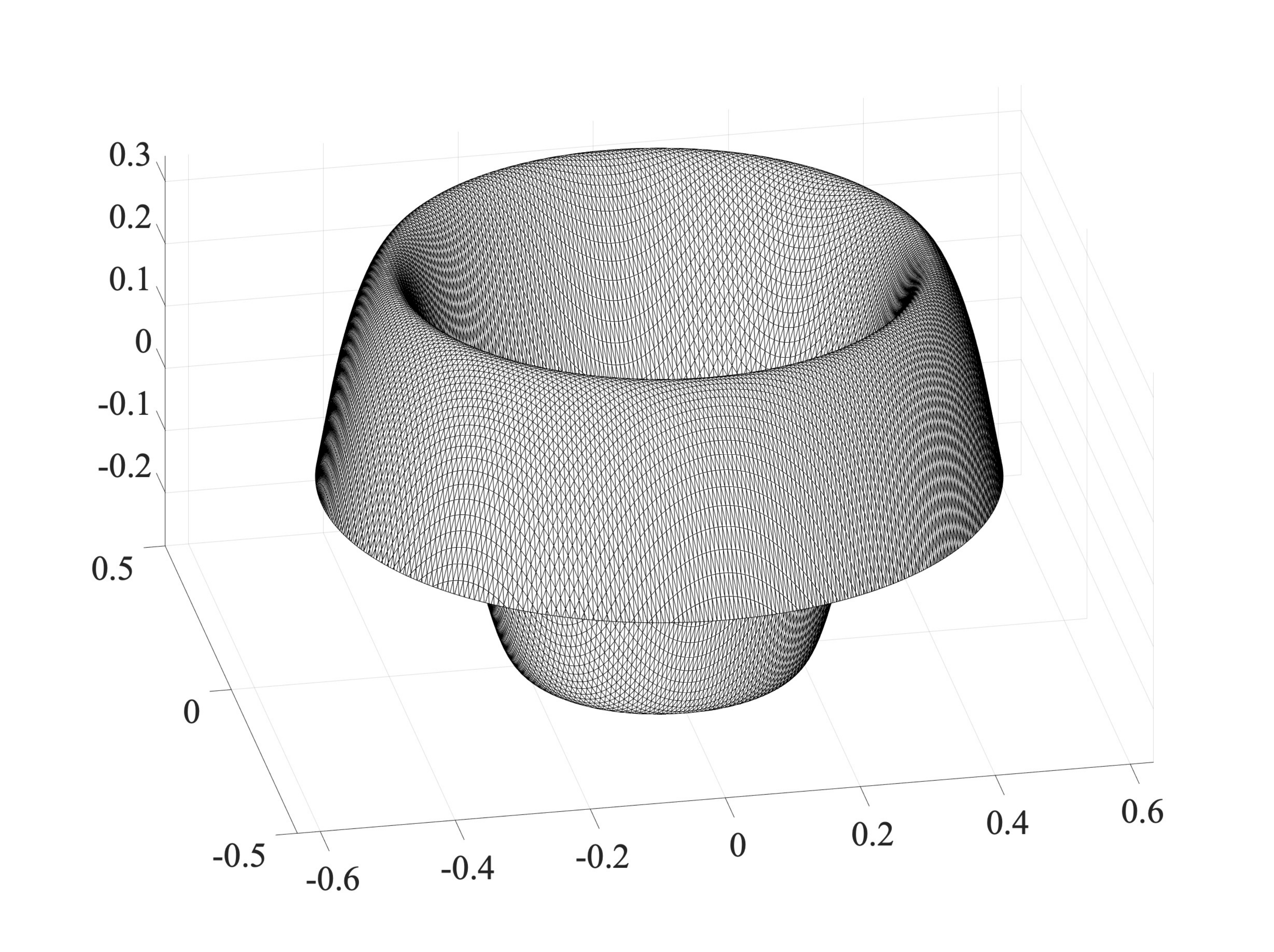}\includegraphics[scale=0.17]{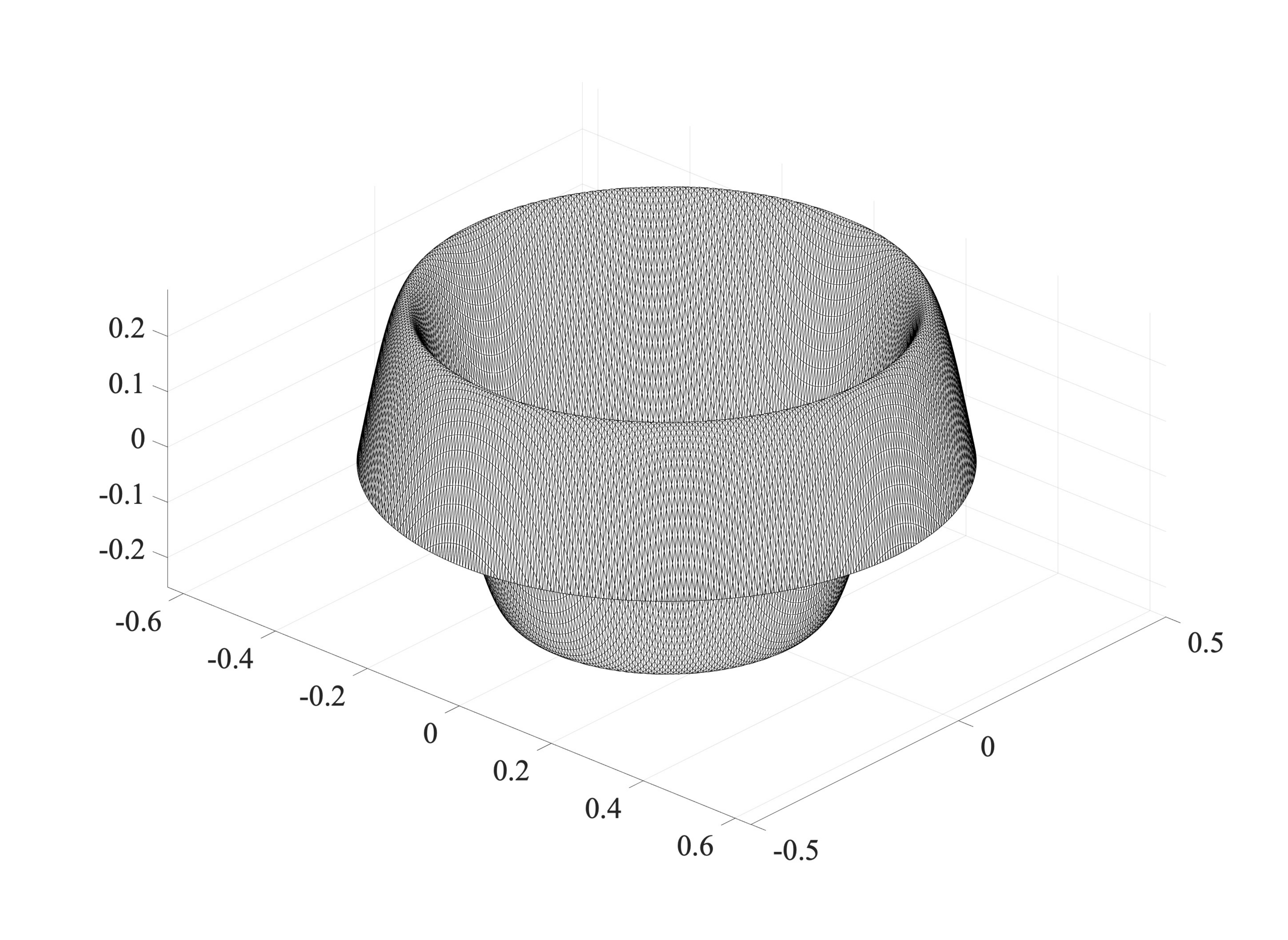}
	\end{center}
	\caption{Pulse at $t=0.35$ (left) and $t=0.4$ (right) for the Dirichlet problem.}
	\label{fig:solD}
\end{figure}

\begin{figure}[ht]
	\begin{center}
		\includegraphics[scale=0.20]{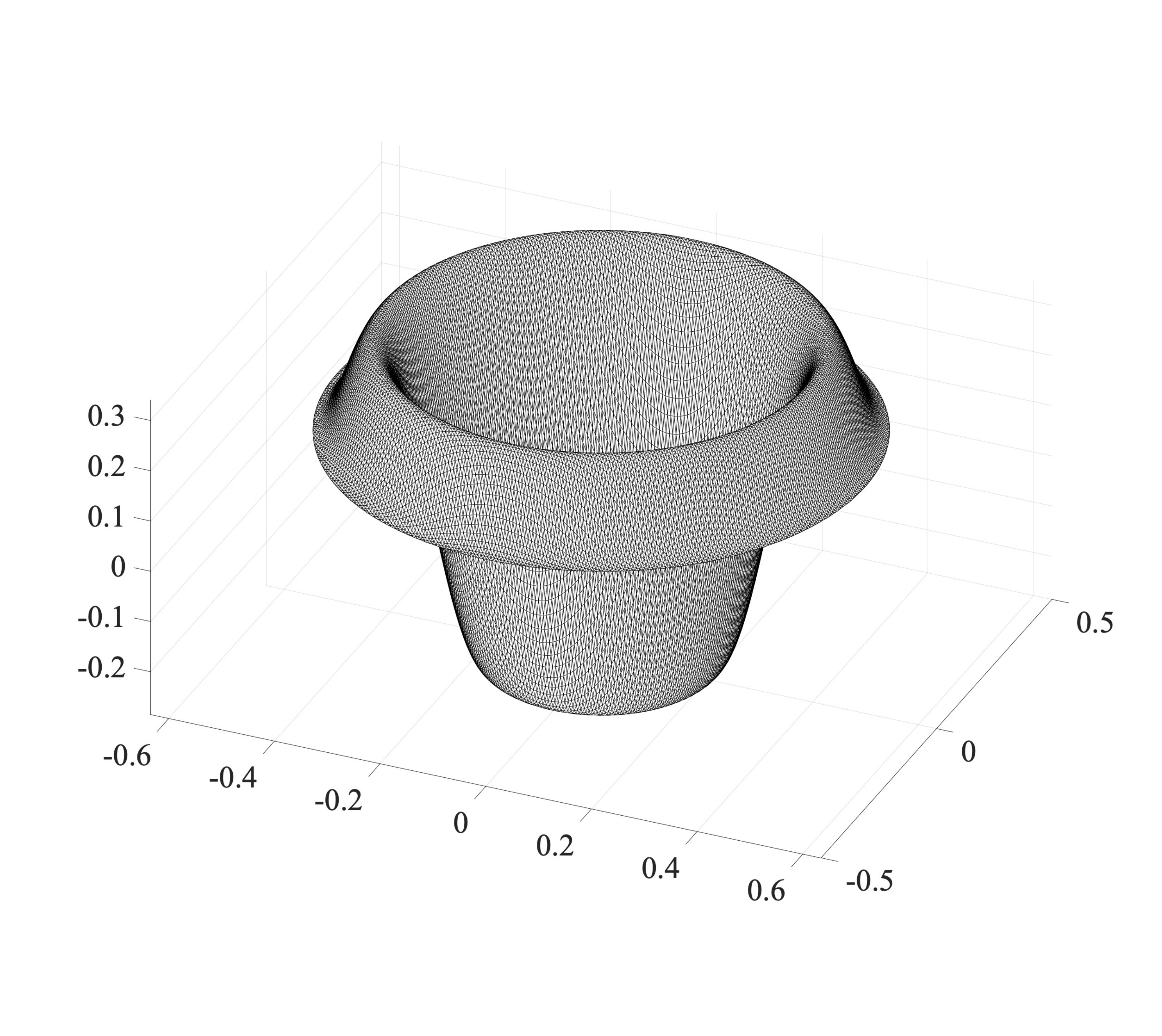}\includegraphics[scale=0.18]{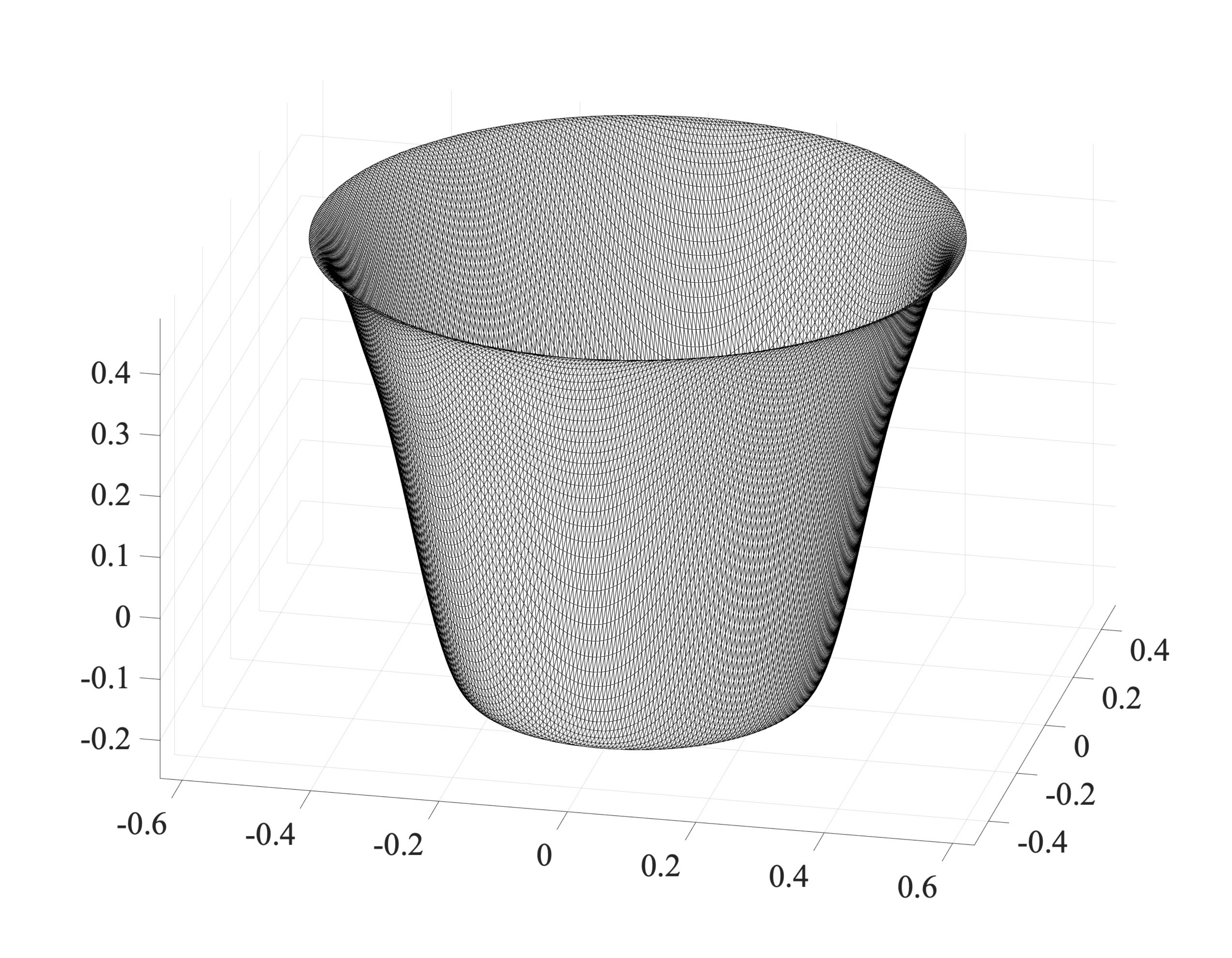}
	\end{center}
	\caption{Pulse at $t=0.35$ (left) and $t=0.4$ (right) for the Neumann problem.}
	\label{fig:solN}
\end{figure}

\begin{figure}[ht]
	\begin{center}
		\includegraphics[scale=0.20]{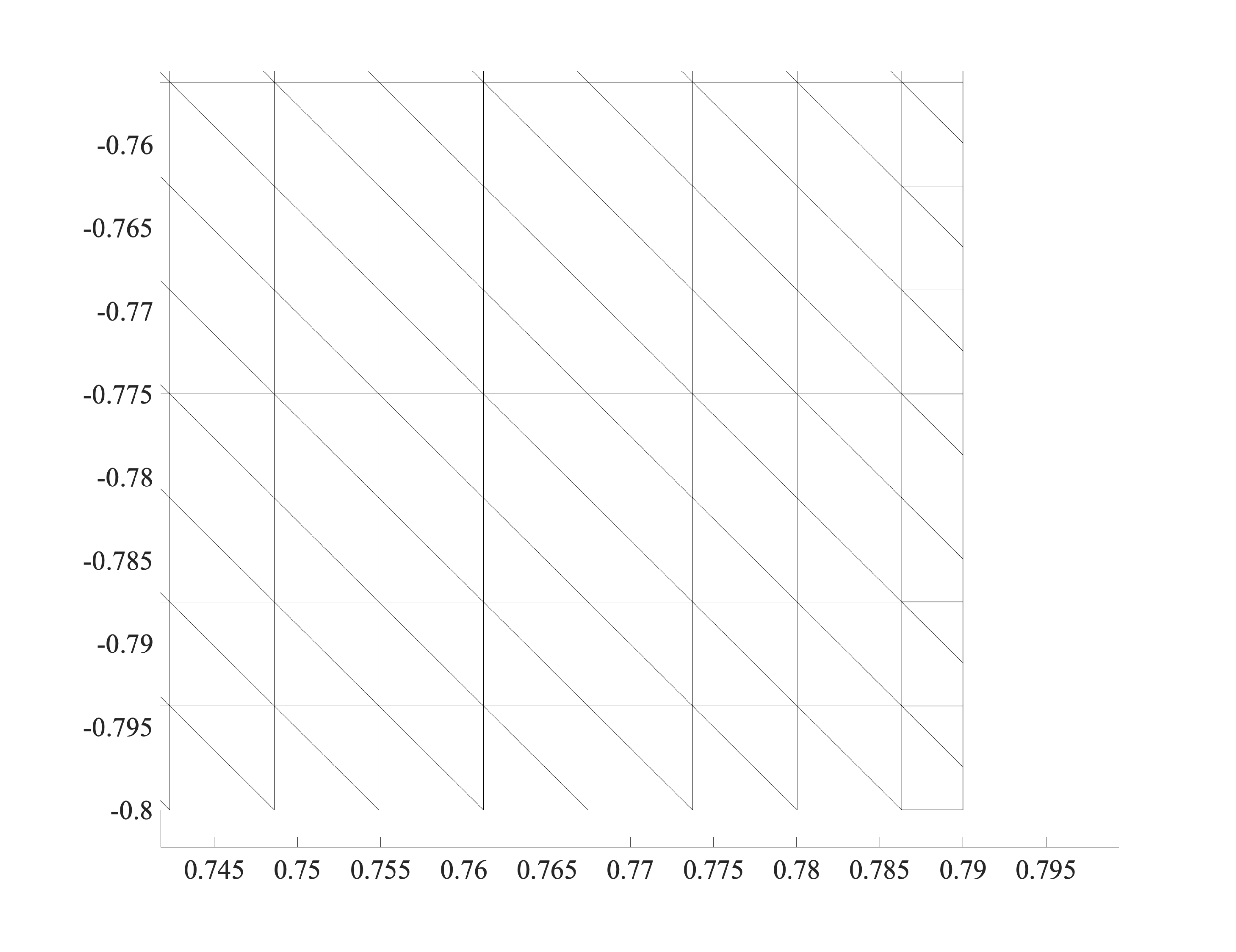}
	\end{center}
	\caption{Closeup of the mesh at the lower right corner.}
	\label{fig:cut}
\end{figure}
\begin{figure}[ht]
	\begin{center}
		\includegraphics[scale=0.20]{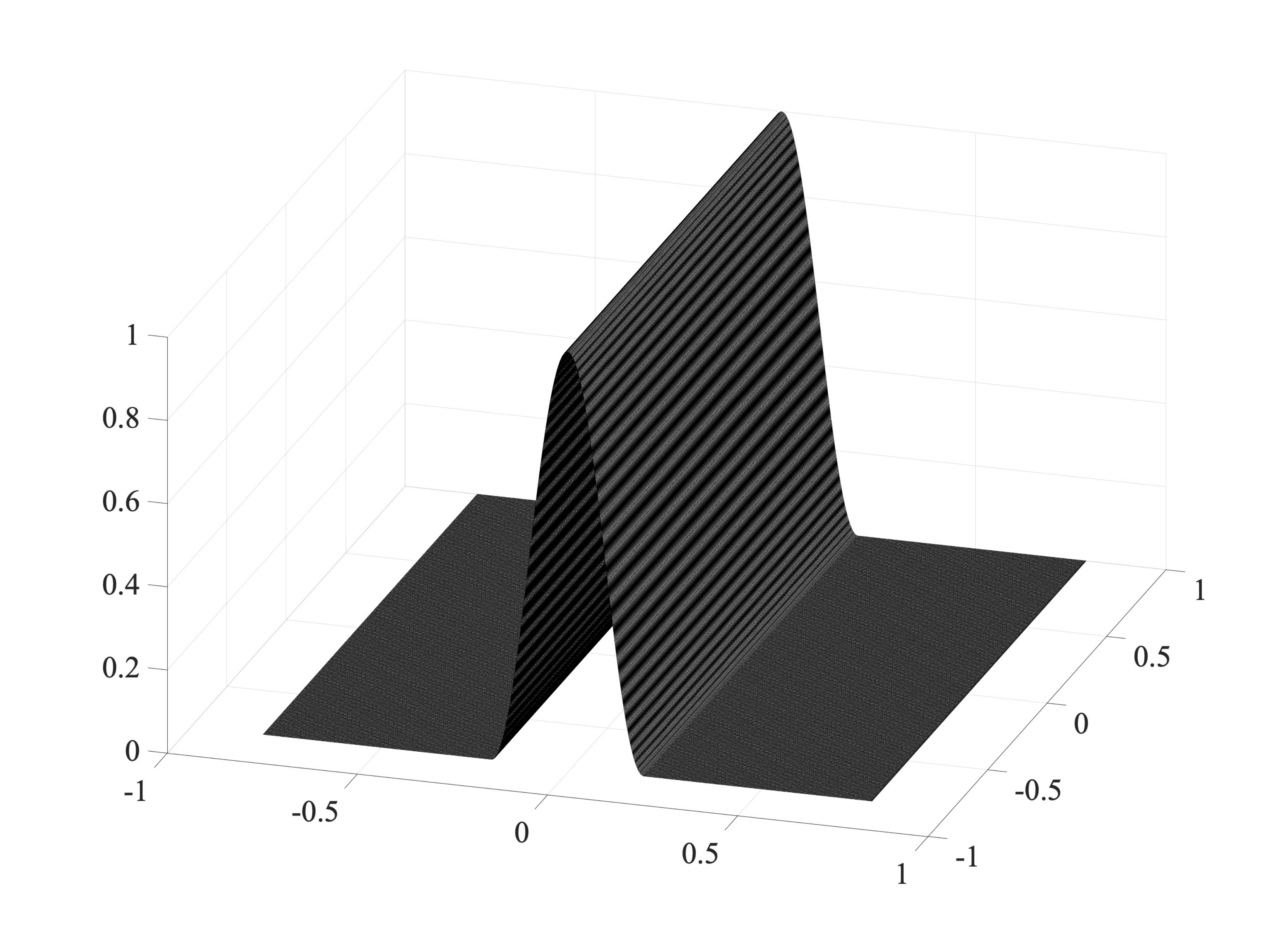}
	\end{center}
	\caption{Pulse at $t=0$ for $d_0=0.2$.}
	\label{fig:pulsestart}
\end{figure}
\begin{figure}[ht]
	\begin{center}
		\includegraphics[scale=0.16]{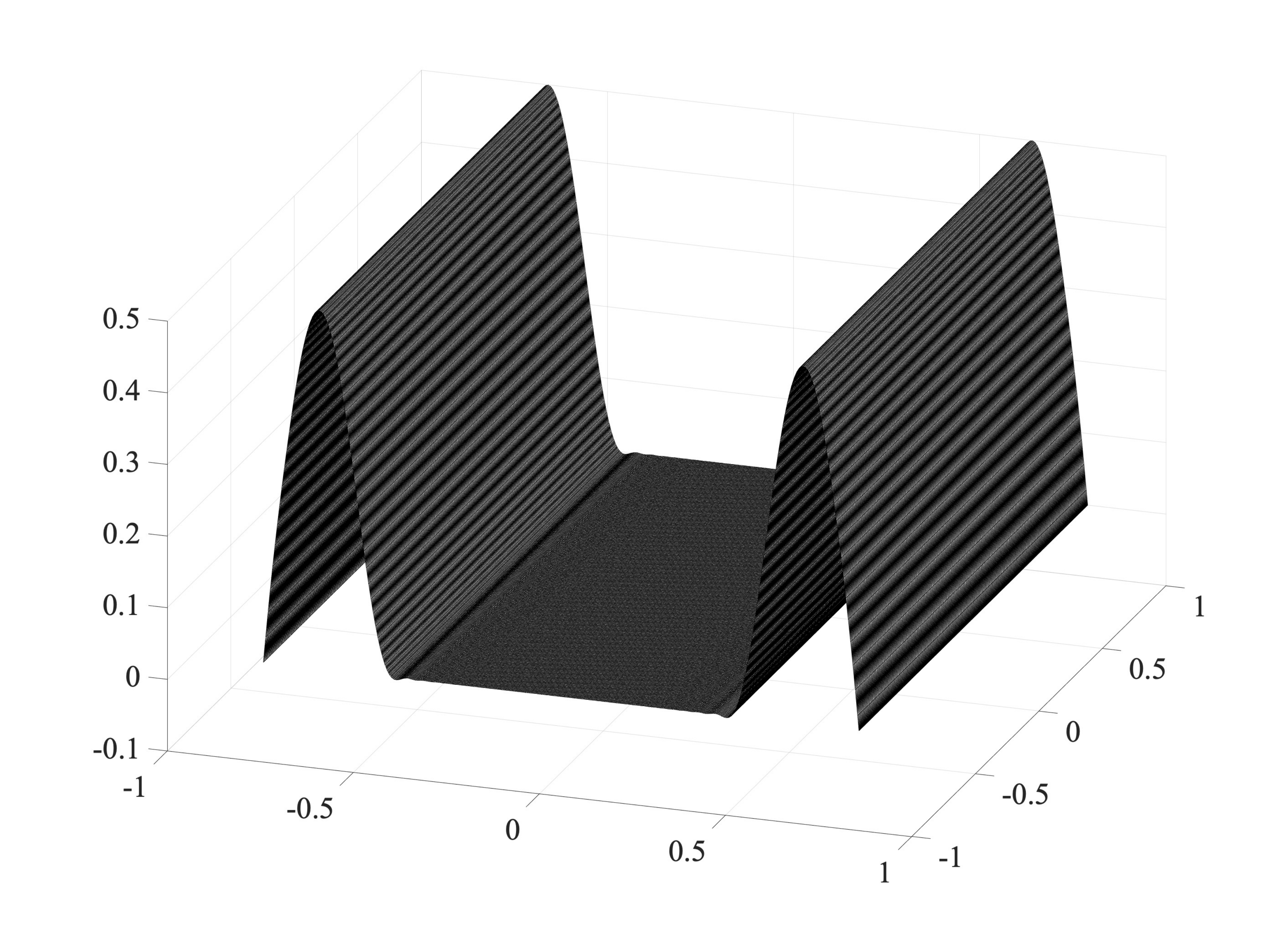}
		\includegraphics[scale=0.16]{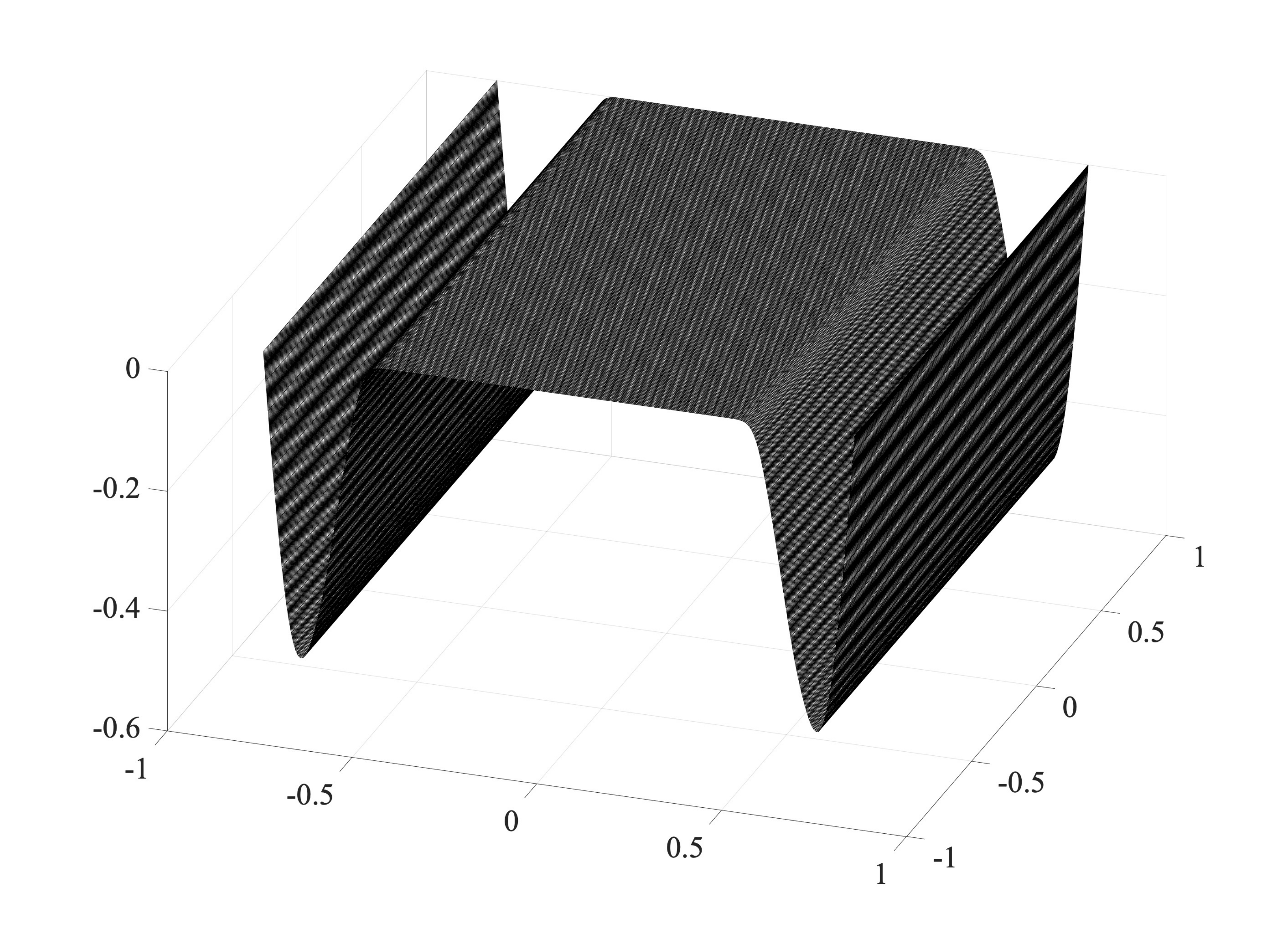}
	\end{center}
	\caption{Pulse at $t=0.65$ (left) and $t=0.9$ (right) for $d_0=0.2$.}
\end{figure}
\begin{figure}[ht]
	\begin{center}
		\includegraphics[scale=0.20]{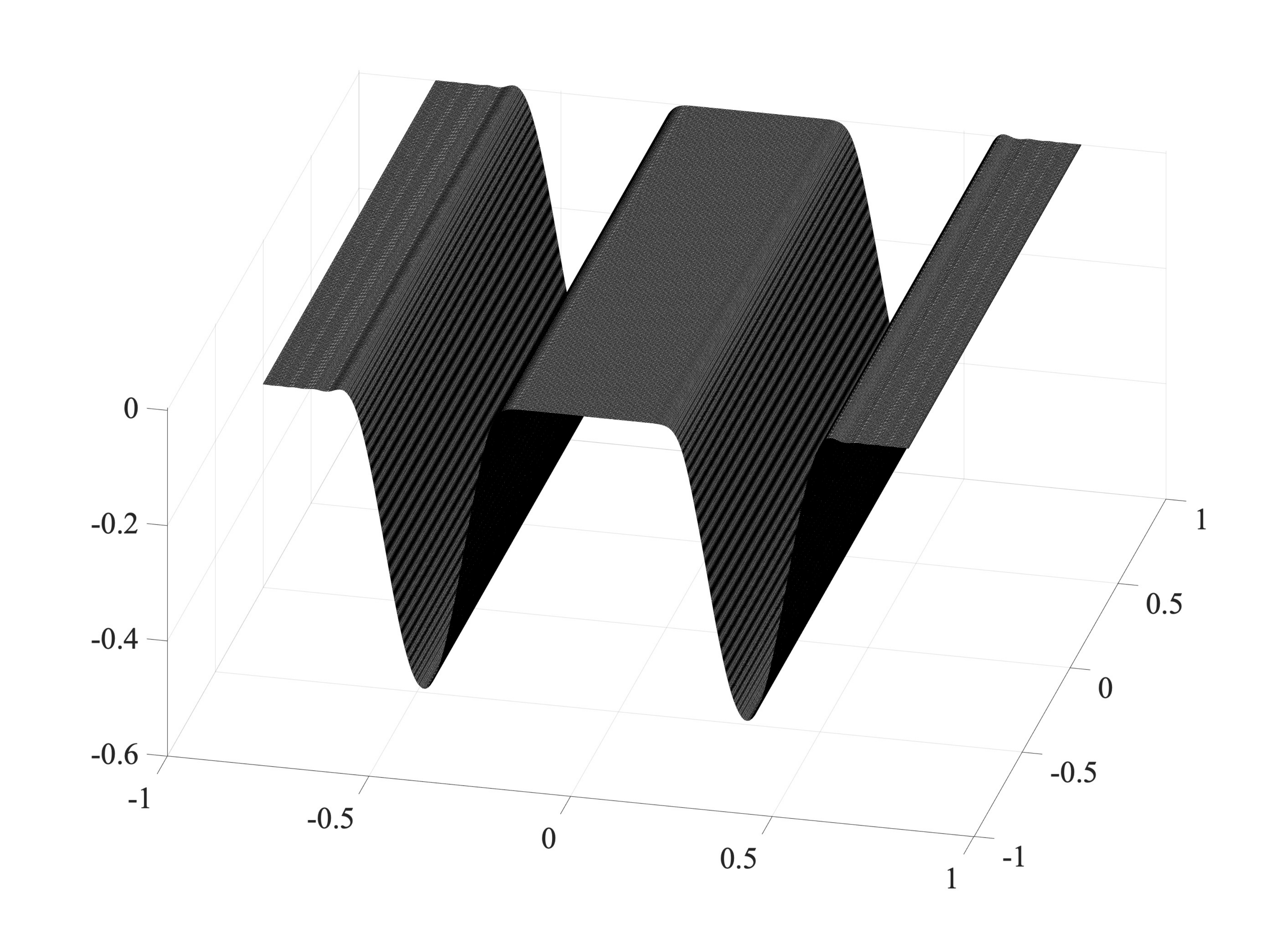}
	\end{center}
	\caption{Pulse at $t=1.2$ for $d_0=0.2$.}
\end{figure}
\begin{figure}[ht]
	\begin{center}
		\includegraphics[scale=0.20]{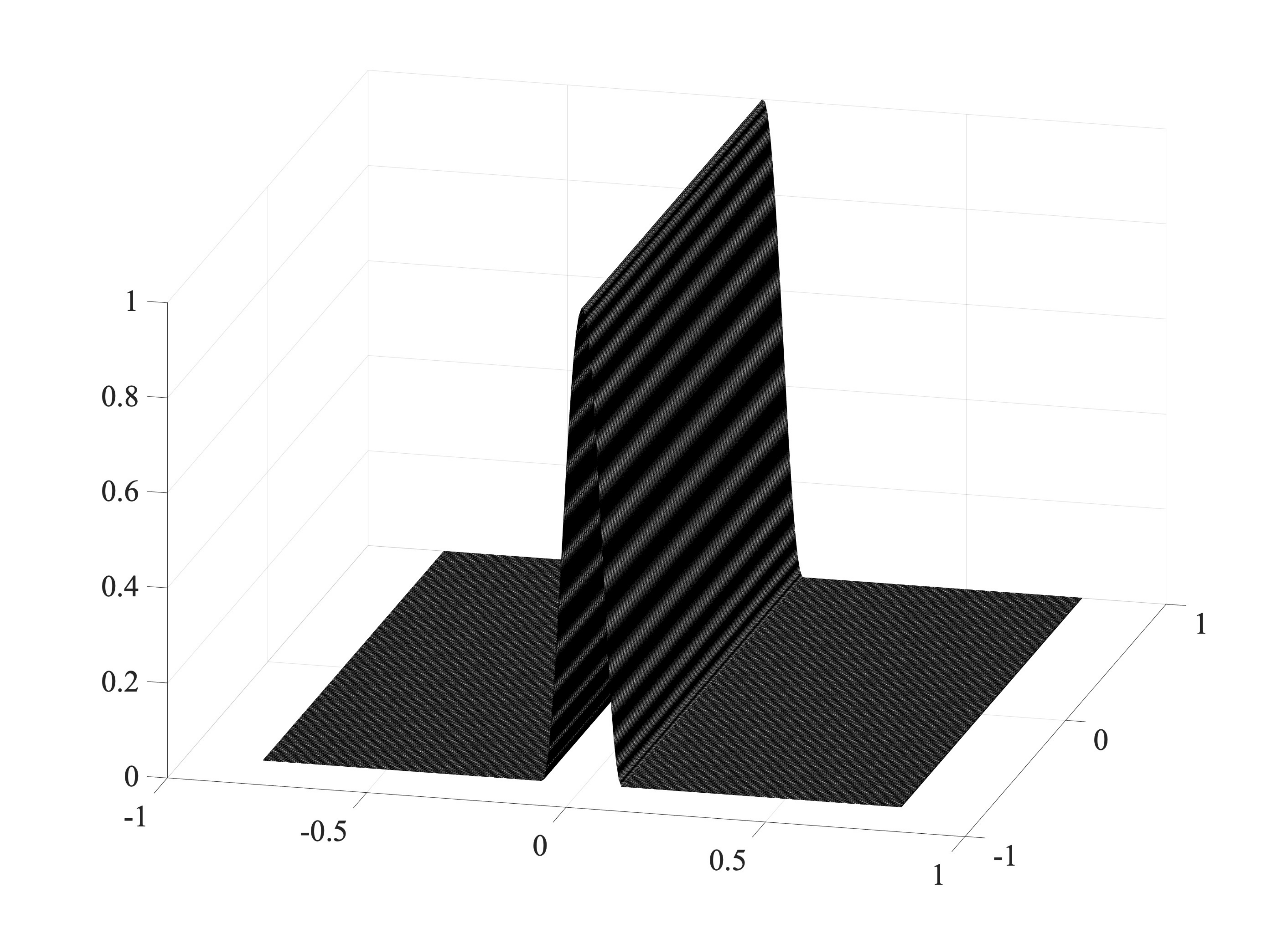}
	\end{center}
	\caption{Pulse at $t=0$ for $d_0=0.1$.}
\end{figure}
\begin{figure}[ht]
	\begin{center}
		\includegraphics[scale=0.16]{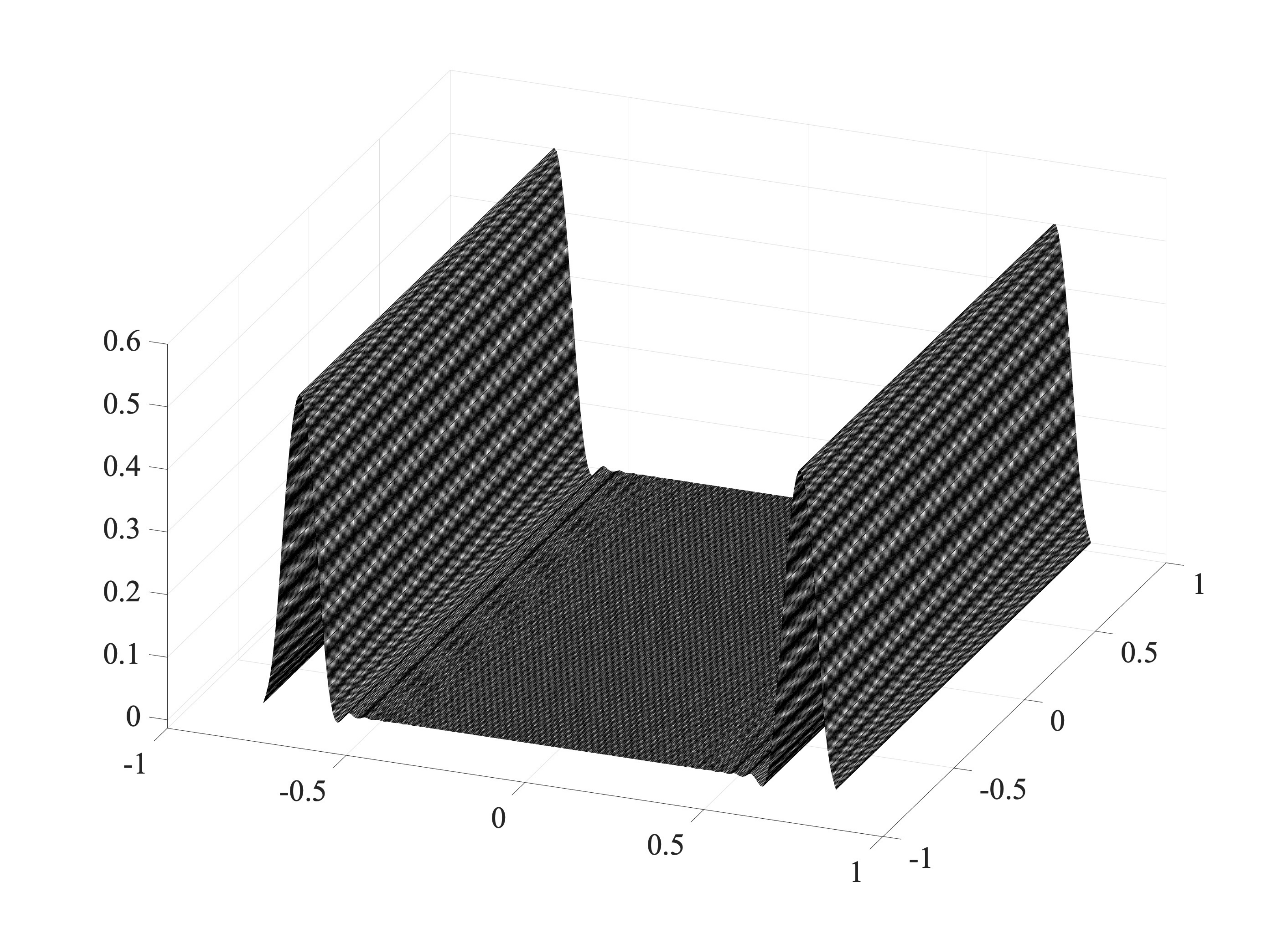}
		\includegraphics[scale=0.16]{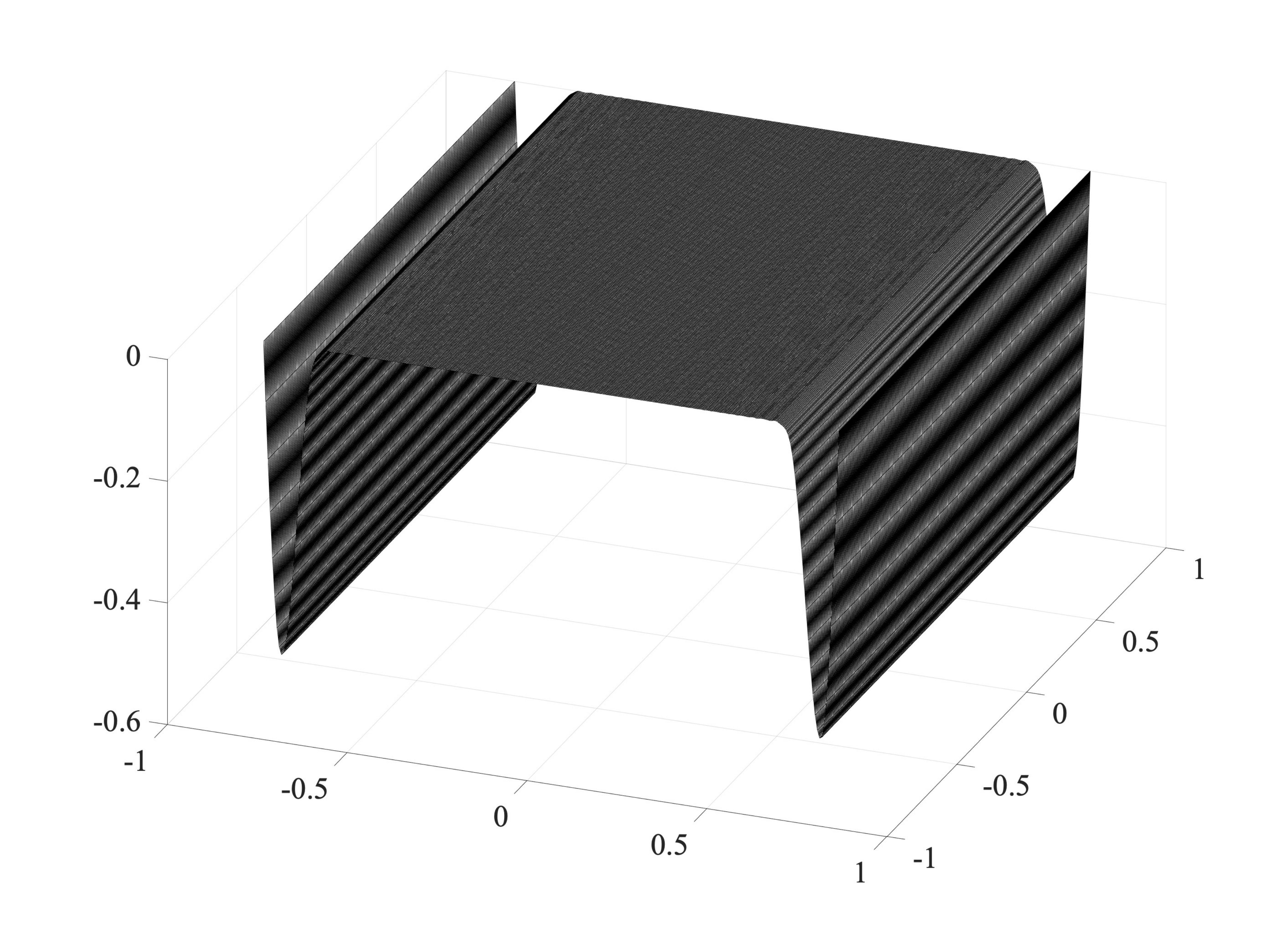}
	\end{center}
	\caption{Pulse at $t=0.7$ (left) and $t=0.85$ (right) for $d_0=0.1$.}
\end{figure}
\begin{figure}[ht]
	\begin{center}
		\includegraphics[scale=0.20]{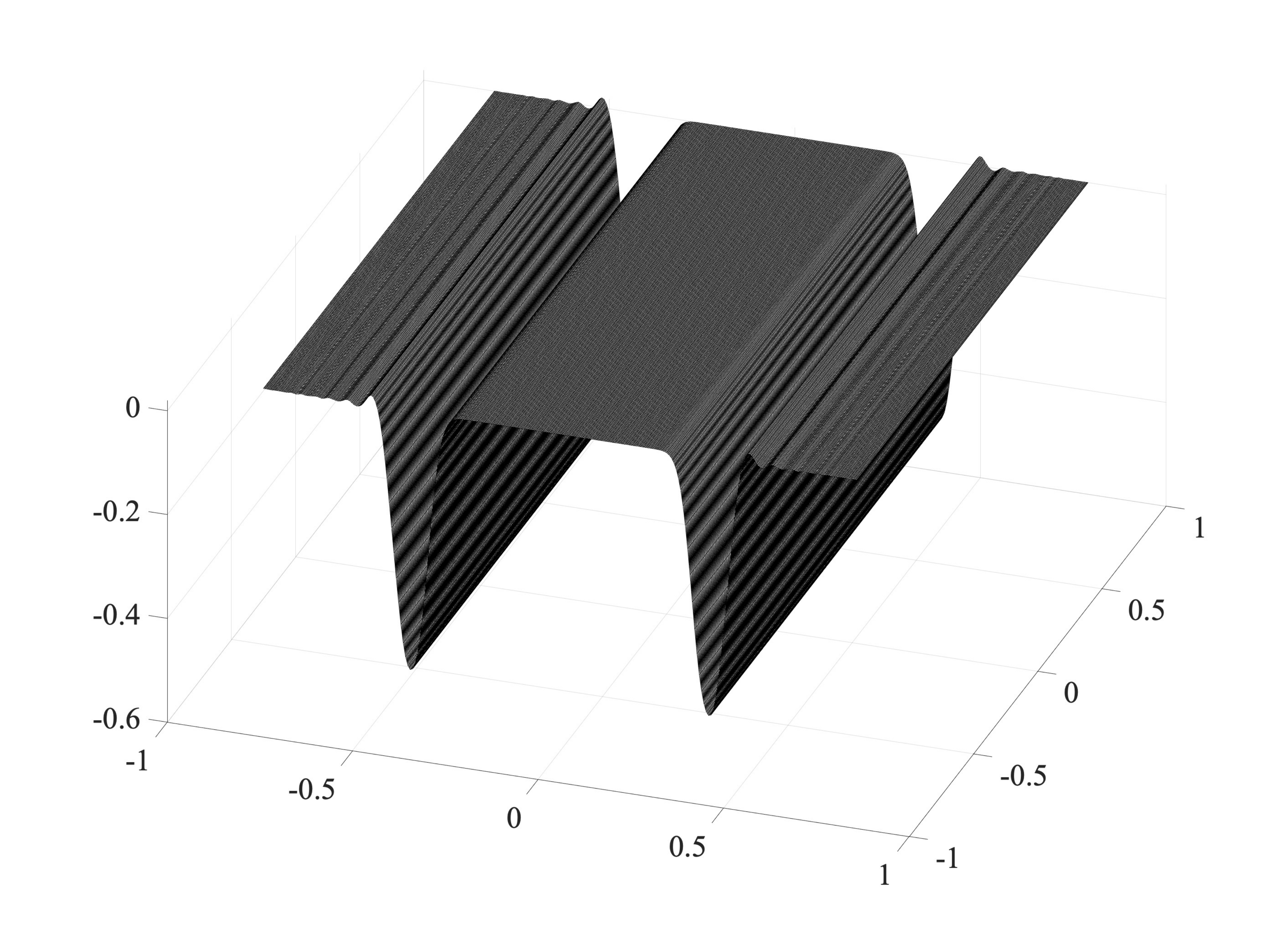}
	\end{center}
	\caption{Pulse at $t=1.2$ for $d_0=0.1$.}
\end{figure}
\begin{figure}[ht]
	\begin{center}
		\includegraphics[scale=0.20]{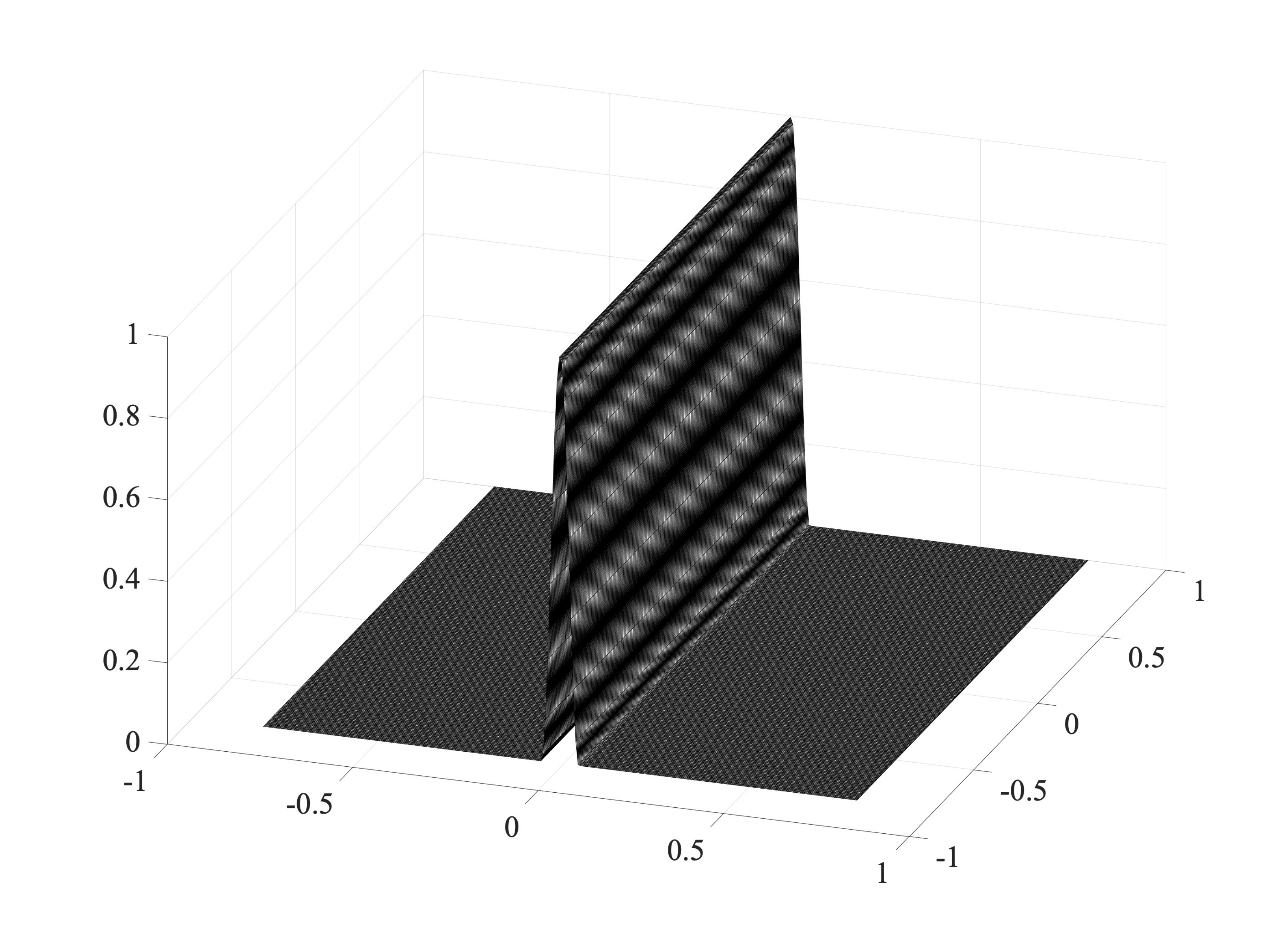}
	\end{center}
	\caption{Pulse at $t=0$ for $d_0=0.05$.}
\end{figure}
\begin{figure}[ht]
	\begin{center}
		\includegraphics[scale=0.16]{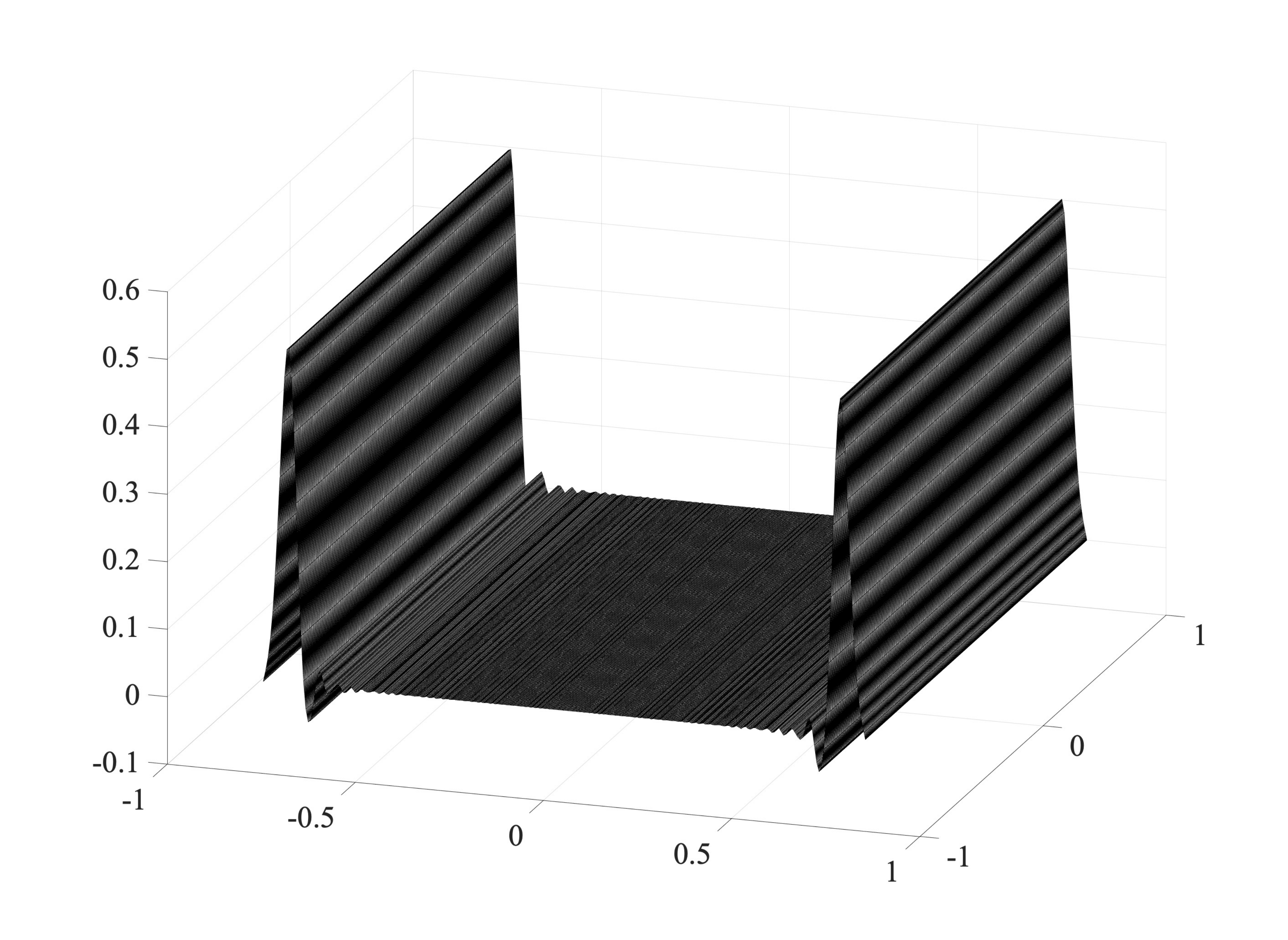}
		\includegraphics[scale=0.16]{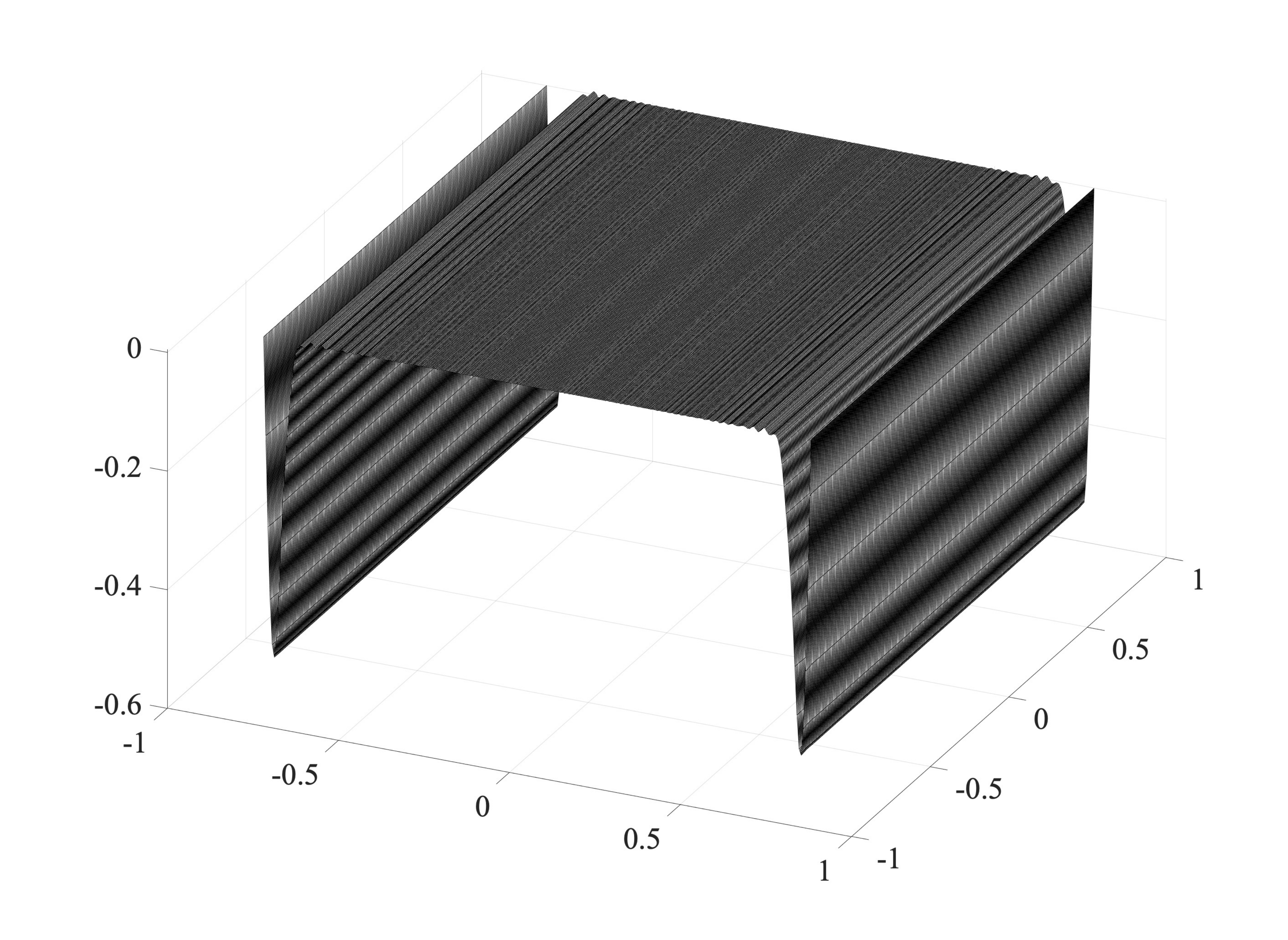}
	\end{center}
	\caption{Pulse at $t=0.74$ (left) and $t=0.84$ (right) for $d_0=0.05$.}
\end{figure}
\begin{figure}[ht]
	\begin{center}
		\includegraphics[scale=0.20]{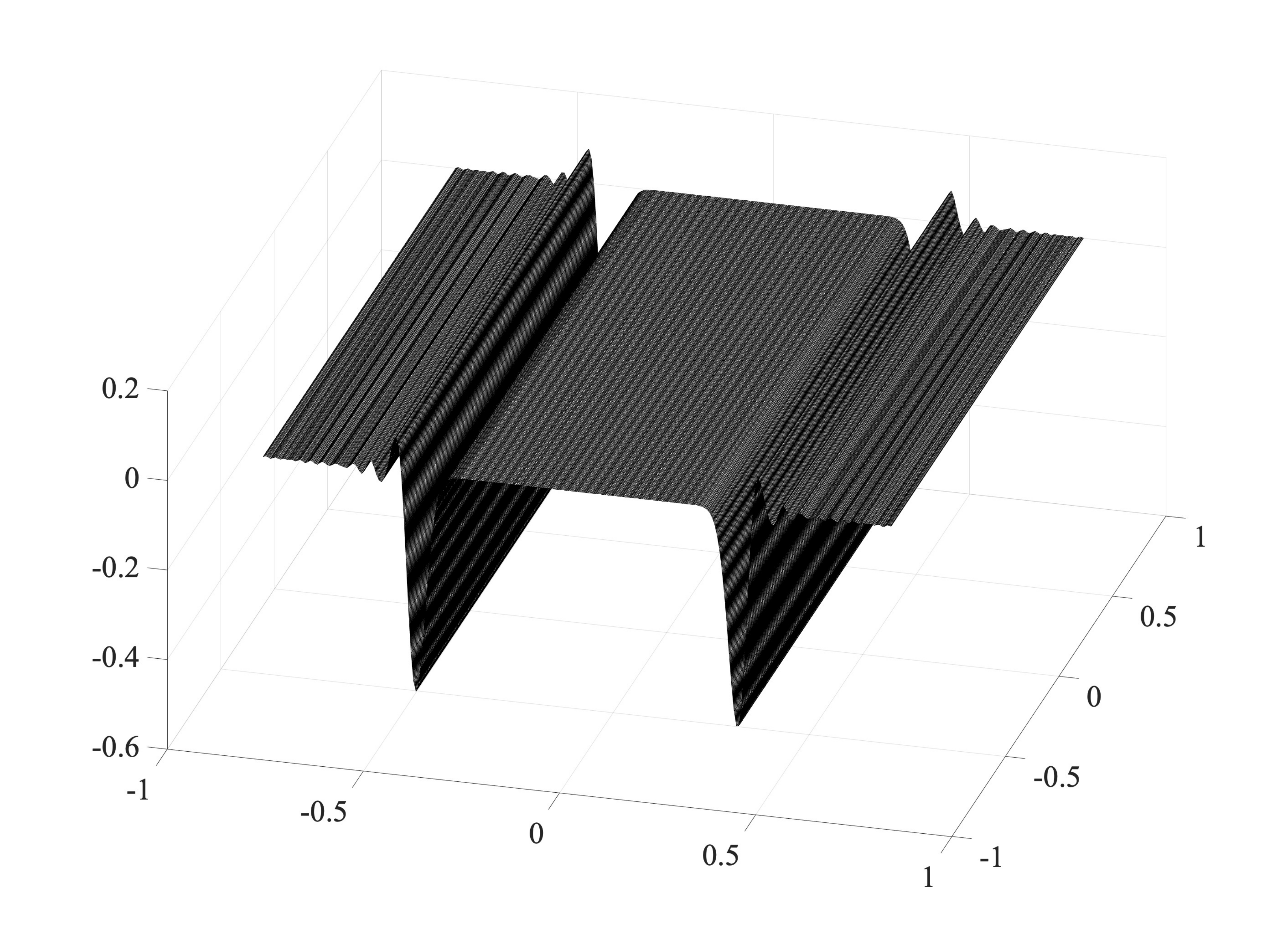}
	\end{center}
	\caption{Pulse at $t=1.2$ for $d_0=0.05$.}
	\label{fig:pulseend}
\end{figure}

\end{document}